\documentclass[12pt]{amsart}

\usepackage{latexsym,amssymb} 
\usepackage{amsmath,amsfonts,amsthm}
\input xypic

\newtheorem{theorem}{Theorem}[section]
\newtheorem{proposition}[theorem]{Proposition}

\newtheorem{lemma}[theorem]{Lemma}

\newtheorem{remark}[theorem]{Remark}
\newtheorem{remarks}[theorem]{Remarks}

\newcommand{\dom}{\mathbf{d}}
\newcommand{\ran}{\mathbf{r}}

\newcommand{\s}{\mathbf{d}}
\newcommand{\f}{\mathbf{r}}
\newcommand{\G}{{\bf G}}
\newcommand{\Le}{{\bf L}}
\newcommand{\R}{{\bf R}}
\newcommand{\D}{\mathcal{D}}

\title[Chapter 2]{Primer on inverse semigroups II}
\author{Mark~V.~Lawson}
\address{Department of Mathematics and the
Maxwell Institute for Mathematical Sciences\\
Heriot-Watt University\\
Riccarton\\
Edinburgh~EH14~4AS\\
Scotland\\
\texttt{M.V.Lawson@hw.ac.uk}} 
 
\begin{document}
\maketitle

\section{Introduction}\setcounter{theorem}{0}

We saw in Chapter~1 that underlying every inverse semigroup is a groupoid, 
but this groupoid is not enough on its own to recover the original inverse semigroup multiplication. 
This raises the question of what else is needed, and we answer this question in Section~2.
The underlying groupoid of an inverse semigroup combined with the natural partial order lead to a structure called an inductive groupoid
and inverse semigroups and inductive groupoids are two ways of looking at the same thing.
In fact, Ehresmann worked with inductive groupoids rather than inverse semigroups.

Our second categorical description takes its cue from how substructures are represented in a category.
This leads to left (or right) cancellative categories and their actions on principal groupoids 
as a way of constructing arbitrary inverse semigroups.
This is described in Section~4.

Section~3 forms a bridge between Sections~2 and 4.
In it, we describe the extent to which ordered groupoids are related to left cancellative categories.
In the case of inverse monoids, this leads to a complete description in terms of categories.

\section{Ordered groupoids}\setcounter{theorem}{0}

The motivation for this construction is described in Section~2.1 and the main theorem,
the Ehresmann-Schein-Nambooripad theorem, is proved in Section~2.3.

\subsection{Motivation}

The following result shows how the usual product in an inverse semigroup can be reconstructed from the restricted product and the natural partial order.

\begin{lemma} Let $S$ be an inverse semigroup.
\begin{enumerate}

\item Let $s \in S$ and $e$ an idempotent 
such that $e \leq  s^{-1}s$.  
Then $a = se$ is the unique element in $S$ such that 
$a \leq  s$ and $a^{-1}a = e$. 

\item Let $s \in S$ and $e$ an idempotent 
such that $e \leq  ss^{-1}$.  
Then $a = es$ is the unique element in $S$ such that $a \leq  s$ 
and $aa^{-1} = e$. 

\item Let $s,t \in  S$. 
Then  $st = s' \cdot t'$  where $s' = se$, $t' = et$
and $e = s^{-1}stt^{-1}$.

\end{enumerate}
\end{lemma}
\begin{proof}
(1) From the definition of the natural partial order
we have that $a \leq s$. 
Also, $a^{-1}a = (se)^{-1}se = es^{-1}se = e$.
Now let $b \leq s$ be such that $b^{-1}b = e$.
Then $b = sb^{-1}b$, so that $b = se = a$.

(2) Similar to (1). 

(3) Put $s' = se$ and $t' = et$ where  $e = s^{-1}stt^{-1}$. 
Then $s' \leq s$ and $t' \leq t$. 
It is easy to check that
${\bf d}(s') = e$ and ${\bf r}(t') = e$.
Thus $s' \cdot t'$ exists. 
But $s'\cdot t' = set = st$.
\end{proof}

A function $\theta \colon \: S \rightarrow  T$ between inverse 
semigroups is said to be a 
{\em prehomomorphism} 
if $\theta (st) \leq  \theta (s)\theta (t)$ 
for all $s,t \in  S$.\footnote{The function is called a 
{\em dual prehomomorphism} 
if
$\theta (s)\theta (t) \leq \theta (st)$.}
Inverse semigroups and prehomomorphisms form a category
that contains the usual category of inverse semigroups and homomorphisms.
We can easily construct examples of prehomomorphisms 
which are not homomorphisms.
Let $L$ and $M$ be meet semilattices and 
let $\theta \colon \: L \rightarrow  M$ be an order-preserving function.  
Let $e,f \in L$.
Then $e \wedge f \leq e,f$ and so 
$\theta (e \wedge f) \leq \theta (e),\theta (f)$ since
$\theta$ is order-preserving.
Thus 
$$\theta (e \wedge f) \leq \theta (e) \wedge \theta (f)$$
since $M$ is a meet semilattice.
It follows that $\theta$ is a prehomomorphism
from the inverse semigroup $(L,\wedge)$ to the 
inverse semigroup $(M,\wedge)$,
but not in general a homomorphism. 

\begin{lemma} Let $\theta \colon \: S \rightarrow T$ be a function between inverse semigroups.
\begin{enumerate}

\item  $\theta$ is a prehomomorphism if, and only if, it preserves the restricted product and  the natural partial order. 

\item  $\theta$ is a homomorphism  if, and only if, it is a prehomomorphism which satisfies $\theta (ef) = \theta (e)\theta (f)$
for all idempotents $e,f \in  S$.

\end{enumerate}
\end{lemma} 
\begin{proof}
(1) Let $\theta \colon \: S \rightarrow T$ be a prehomomorphism.
We first prove that 
$\theta (s^{-1}) = \theta (s)^{-1}$ for each $s \in S$. 
By definition 
$$\theta (s) 
= \theta (s(s^{-1}s)) \leq  \theta (s)\theta (s^{-1}s) 
\leq  \theta (s)\theta (s^{-1})\theta (s).$$ 
Similarly,
$$\theta (s^{-1}) \leq  \theta (s^{-1})\theta (s)\theta (s^{-1}).$$
Put $a = \theta (s)$ and $b = \theta (s^{-1})$.  
Then $a \leq  aba$ and $b \leq  bab$.  
Now $ab \leq  abab$, so that
$$ab = (ab)^{2}(ab)^{-1}(ab) = (ab)^{2}.$$
Similarly, $ba = (ba)^{2}$.  
Thus $a(ba) \leq  a \leq  aba$, and so $a = aba$.  
Similarly, $b = bab$.  
Hence $\theta (s^{-1}) = \theta (s)^{-1}$. 

Next we show that if $e$ is an idempotent then $\theta (e)$ is an idempotent. 
Let $e$ be an idempotent.  
Then $\theta (e) = \theta (e^{-1}) = \theta (e)^{-1}$ by the result above.  
Thus
$$\theta (e) = \theta (ee) \leq \theta (e)\theta (e) 
= \theta (e)\theta (e^{-1}e) \leq  \theta (e)\theta (e)^{-1}\theta (e) 
= \theta (e),$$
and so $\theta (e) = \theta (e)\theta (e)$. 

We can now prove that $\theta$ is order-preserving.
Let $s \leq  t$.  
Then $s = te$ for some idempotent $e$.  
Thus
$$\theta (s) = \theta (te) \leq  \theta (t)\theta (e) \leq  \theta (t)$$
since $\theta (e)$ is an idempotent. 

A key ingredient in proving that $\theta$ preserves restricted
products is the following:
$$\theta (ss^{-1}) = \theta (s) \theta (s)^{-1} 
\mbox{ and }
\theta (s^{-1}s) = \theta (s)^{-1} \theta (s)$$ 
for every $s \in S$.
We show that $\theta (ss^{-1}) = \theta (s)\theta (s)^{-1}$;  
the proof of the other case is similar.
Clearly, $\theta (ss^{-1}) \leq  \theta (s)\theta (s)^{-1}$,  
so that
$$\theta (ss^{-1}) \theta (s) \leq  \theta (s)\theta (s)^{-1}\theta (s) 
= \theta (s).$$
But then
$$\theta (s) = \theta ((ss^{-1})s) \leq  \theta (ss^{-1})\theta (s) 
\leq  \theta (s).$$
Thus $\theta (s) = \theta (ss^{-1})\theta (s)$, and so 
$$\theta (s)\theta (s)^{-1} = \theta (ss^{-1})(\theta (s)\theta (s)^{-1}).$$
Hence $\theta (s)\theta (s)^{-1} \leq  \theta (ss^{-1})$.  
But $\theta (ss^{-1}) \leq  \theta (s)\theta (s)^{-1}$.  
It follows that $\theta (s)\theta (s)^{-1} = \theta (ss^{-1})$. 

We can now prove that $\theta$ preserves restricted products.
Suppose that $s \cdot t$ is defined. 
Then by the result above so too is $\theta (s) \cdot \theta (t)$.
It remains to show that $\theta (s \cdot t) = \theta (s) \cdot \theta (t)$.  
Clearly $\theta (s \cdot t) \leq  \theta (s) \cdot \theta (t)$.  
Now
$$\theta (s \cdot t)^{-1} \theta (s \cdot t) 
= \theta ((s \cdot t)^{-1}(s \cdot t)) = \theta (t^{-1}s^{-1}st) 
= \theta (t^{-1}t),$$
and
$$[\theta (s)\cdot \theta (t)]^{-1}[\theta (s) \cdot \theta (t)] 
= \theta (t)^{-1}\theta (t) = \theta (t^{-1}t).$$
Hence $\theta (s \cdot t) = \theta (s) \cdot \theta (t)$ as required. 

To prove the converse, suppose that $\theta$ preserves the 
restricted product and the natural partial order.
We show that it is a prehomomorphism.
Let $st$ be a full product in $S$.  
Then $st = (se)\cdot (et)$ where $e = s^{-1}stt^{-1}$ by Lemma~2.1.  
Thus, by assumption, $\theta (st) = \theta (se)\cdot \theta (et)$.  
But $se \leq  s$ and $et \leq  t$ so that 
$\theta (se) \leq  \theta (s)$ and 
$\theta (et) \leq  \theta (t)$. 
Hence $\theta (st) \leq  \theta (s)\theta (t)$ as required. 

(2) We now prove that 
if $\theta \colon \: S \rightarrow  T$ 
is a prehomomorphism satisfying
$\theta (ef) = \theta (e)\theta (f)$ 
for all idempotents $e,f \in S$, then $\theta $ is a homomorphism.  
Let $st$ be a full product in $S$.  
Then $st = (se)\cdot (et)$ where $e = s^{-1}stt^{-1}$.  
Thus $\theta (st) = \theta (se)\cdot \theta (et)$.  
We show that $\theta (se) = \theta (s)\theta (e)$.  
Clearly, $\theta (se) \leq  \theta (s)\theta (e)$.  
Now
$$\theta (se)^{-1}\theta (se) = \theta ((se)^{-1}(se)) 
= \theta (es^{-1}se) = \theta (e),$$
and $$[\theta (s) \theta (e)]^{-1} \theta (s) \theta (e) 
= \theta (e)^{-1} \theta (s)^{-1} \theta (s) \theta (e) 
= \theta (e)^{-1} \theta (s^{-1}s)\theta (e)  
= \theta (e)$$
since $\theta (e) = \theta (s^{-1}s) \theta (tt^{-1})$ by assumption.
Thus $\theta (se) = \theta (s)\theta (e)$.
Similarly, $\theta (et) = \theta (e)\theta (t)$.
It now follows that
$$\theta (st) = \theta (s)\theta(e)\theta(e)\theta(t)
= \theta (s)\theta (s^{-1}s)\theta (tt^{-1})\theta (t)
= \theta (s)\theta(t).$$
\end{proof}

Lemma~2.2 implies that every prehomomorphism between inverse semigroups induces a functor between their associated groupoids.

\subsection{Inductive groupoids}

Let $(G,\cdot)$ be a groupoid, and let $\leq$ be a partial order defined on $G$.  
Then $(G,\cdot,\leq )$ is an {\em ordered groupoid} if the following axioms hold: 

\begin{description}

\item[{\rm (OG1)}] $x \leq  y$ implies $x^{-1} \leq  y^{-1}$ 
for all $x,y \in G$. 

\item[{\rm (OG2)}] For all $x,y,u,v \in G$, 
if $x \leq  y, u \leq  v, \exists xu$ and $\exists yv$ then
$xu \leq  yv$. 

\item[{\rm (OG3)}] Let $x \in G$ and let $e$ be an identity 
such that $e \leq  {\bf d}(x)$. 
Then there exists a unique element $(x\, | \,e)$, 
called the {\it restriction of}\index{restriction in an ordered groupoid} 
$x$ {\it to} $e$, such that $(x\, | \,e) \leq  x$ and ${\bf d}(x\,|\,e) = e$. 

\item[{\rm (OG3)$^{\ast}$}] Let $x \in G$ and let $e$ be an identity 
such that $e \leq  {\bf r}(x)$. 
Then there exists a unique  element $(e \, | \, x)$,  
called  the  {\it corestriction  of}\index{corestriction in an order\-ed group\-oid}  
$x$  {\it to}  $e$,  
such  that  $(e \, | \, x) \leq  x$  and ${\bf r}(e\, | \,x) = e$. 

\end{description}

An ordered groupoid is said to be {\em inductive} if the partially ordered set of identities forms a meet-semilattice.
This term was used by Ehresmann to refer to a more restricted class of ordered groupoids than we have defined, but the terminology is now well-established.
As we shall see, if the groupoid is actually a group the order degenerates to equality and so ordered groupoids do not generalize ordered groups.

A functor between two ordered groupoids is said to 
be {\em ordered}
if it is order-preserving.  
An ordered functor between two inductive groupoids is said to be 
{\em inductive}\ 
if it preserves the meet operation on the set of identities.  
An {\em isomorphism} of ordered groupoids is a bijective ordered
functor whose inverse is an ordered functor.

\begin{lemma} Let $\theta \colon G \rightarrow H$ be an ordered
functor between ordered groupoids. 
\begin{enumerate}

\item If $(x \, | \, e)$ is defined in $G$ 
then $(\theta (x) \, | \, \theta (e))$ is defined in $H$ and 
$\theta (x \, | \, e) = (\theta (x) \, | \, \theta (e))$.

\item If $(e \, | \, x)$ is defined in $G$ then 
$(\theta (e) \, | \, \theta (x))$ is defined in $H$ and 
$\theta (e \, | \, x) = (\theta (e) \, | \, \theta (x))$.

\end{enumerate}
\end{lemma}
\begin{proof}
We shall prove (1); the proof of (2) is similar.
By definition $(x \, | \, e) \leq x$ and so 
$\theta (x \, | \, e) \leq \theta (x)$
since $\theta$ is an ordered functor. 
But 
$${\bf d}(\theta (x \, | \, e)) 
= \theta ({\bf d}(x \, | \, e)) = \theta (e)$$
since $\theta$ is a functor.
But by axiom (OG3), $(\theta (x) \, | \, \theta (e))$ is the unique element
less than $\theta (x)$ and with domain $\theta (e)$.
Thus $\theta (x \, | \, e) = (\theta (x) \, | \, \theta (e))$.
\end{proof}

We now establish some of the basic properties of ordered groupoids.

\begin{lemma} Let $(G,\cdot,\leq)$ be an ordered groupoid.  
\begin{enumerate}

\item[{\rm (1)}] If $x \leq  y$ then ${\bf d}(x) \leq  {\bf d}(y)$ 
and ${\bf r}(x) \leq  {\bf r}(y)$. 

\item[{\rm (2)}] The order $\leq $ restricted to {\rm hom-sets} is trivial. 

\item[{\rm (3)}] If $\exists xy$ and $e$ is an identity such that 
$e \leq  {\bf d}(xy)$ then 
$$(xy \, | \, e) = (x \, | \, {\bf r} ( y \, | \, e))(y \, | \, e).$$ 

\item[{\rm (4)}] If $\exists xy$ and $e$ is an identity such that 
$e \leq  {\bf r}(xy)$ then 
$$(e \, | \, xy) = (e \, | \, x)({\bf d}(e \, | \, x)\, | \, y).$$ 

\item[{\rm (5)}] If $z \leq  xy$ then there exist elements 
$x'$ and $y'$ such that $\exists x'y'$,  $x' \leq  x$, $y' \leq  y$ 
and $z = x'y'$. 

\item[{\rm (6)}] Axiom {\em (OG3)}$^*$ is a consequence of the axioms
{\em (OG1)} and {\em (OG3)}. 

\item[{\rm (7)}] The set of identities $G_{o}$ is an order ideal of $G$. 

\item[{\rm (8)}] If  $f \leq  e \leq  {\bf d}(x)$ then 
$(x \, | \, f) \leq  (x \, | \, e) \leq x$. 

\item[{\rm (9)}] If $f \leq  e \leq  {\bf r}(x)$ then 
$(f \, | \, x) \leq  (e \, | \, x) \leq x$.

\item[{\rm (10)}] Let $x,y,e,f \in G$ such that $x \leq y$, $f \leq e$,
$f \leq {\bf d}(x)$ and $e \leq {\bf d}(y)$.
Then $(x \, | \, f) \leq (y \, | \, e)$.

\end{enumerate}
\end{lemma}
\begin{proof}

(1) This is immediate from axioms (OG1) and (OG2). 

(2) Suppose that 
$${\bf d}(x) = {\bf d}(y), \, {\bf r}(x) = {\bf r}(y) 
\mbox{ and } x \leq  y.$$  
In particular, ${\bf d}(x) \leq  {\bf d}(y)$ 
and so by axiom (OG3) there is a unique element $(y\, | \, {\bf d}(x))$ 
such that 
$$(y\, | \, {\bf d}(x)) \leq  y 
\mbox{ and } 
{\bf d}(y\, | \, {\bf d}(x)) = {\bf d}(x).$$  
But the element $x$ also has the property that 
$x \leq  y$ and ${\bf d}(x) = {\bf d}(x)$.  
Thus by uniqueness $(y\, | \, {\bf d}(x)) = x$.  
However, $(y \, | \, {\bf d}(y)) = x$ 
since ${\bf d}(x) = {\bf d}(y)$.
But $(y\, | \, {\bf d}(y)) = y$.  
Hence $x = y$.

(3) Since $e \leq  {\bf d}(xy) = {\bf d}(y)$ 
the restriction $(y\, | \, e)$ is defined.  
Since $(y\, | \, e) \leq  y$ we have that 
${\bf r}(y\, | \, e) \leq  {\bf r}(y) = {\bf d}(x)$.  
Thus $(x \, | \, {\bf r} (y\, | \, e))$ exists and the product 
$(x\, | \, {\bf r}(y\, | \, e))(y \, | \, e)$ exists.  
Clearly, $(x\, | \, {\bf r}(y\, | \, e))(y \, | \, e) \leq  xy$.   
But
$${\bf d}((x\, | \, {\bf r}(y\, | \, e))(y \, | \, e)) = e,$$
\noindent
and so $(x\, | \, {\bf r}(y \, | \, e)) (y \, | \, e) = (xy\, | \, e)$.

(4) Similar to the proof of (3). 

(5) Let $z \leq  xy$.  
Then ${\bf d}(z) \leq  {\bf d}(xy)$ by (1).  
Thus $(xy\, | \, {\bf d}(z))$ exists.  
Now ${\bf d}(xy\, | \, {\bf d}(z)) = {\bf d}(z)$ 
and $(xy\, | \, {\bf d}(z)) \leq xy$, so that $z = (xy\, | \, {\bf d}(z))$.  
By (3),
$$(xy\, | \, {\bf d}(z)) 
= (x\, | \, {\bf r}(y \, | \, {\bf d}(z)))(y \, | \, {\bf d}(z)).$$
Put $x'  = (x\, | \, {\bf r}(y \, | \, {\bf d}(z)))$ 
and $y'  = (y\, | \, {\bf d}(z))$, and we have the result. 

(6) Suppose that the axioms (OG1) and (OG3) hold.  
We show that axiom (OG3)$^*$ holds.  
Let $e \leq  {\bf r}(x)$.  
Then $e \leq  {\bf d}(x^{-1})$, 
so that $(x^{-1}\, | \, e)$ exists by axiom (OG3).  
Define $(e\, | \, x) = (x^{-1}\, | \, e)^{-1}$.  
Then  $(e \,| \,x) \leq  x$ by axiom (OG1),
and ${\bf r}(e \,|\, x) = {\bf d}(x^{-1}\, | \, e) = e$.  
Now for uniqueness.  
Suppose that $y \leq  x$ and ${\bf r}(y) = e$.  
Then $y^{-1} \leq  x^{-1}$ by axiom (OG1)
and ${\bf d}(y^{-1}) = e$. 
Thus $y^{-1} = (x^{-1}\, | \, e)$ by (OG3), 
and so $y = (e\, | \, x)$ by (OG1). 

(7) Let $x \leq  e$ where $e$ is an identity.  
Then ${\bf d}(x) \leq  e$. 
But $x,{\bf d}(x) \leq e$ and ${\bf d}(x) = {\bf d}({\bf d}(x))$.
Thus $x = {\bf d}(x)$. 

(8) Let $f \leq  e \leq  {\bf d}(x)$.  
Both $(x\, | \, e)$ and $(x\, | \, f)$ exist.  
Now $f \leq  {\bf d}(x\, | \, e)$ and so the element 
$((x\, | \, e) \, |\, f)$ exists.  
But ${\bf d}(x \, | \, f) = f$ and $(x \, | \, f) \leq  x$.  
Thus $((x \, | \, e)\, | \, f) = (x\, | \, f)$, 
and so $(x\, | \, f) \leq (x \, | \, e)$. 

(9) Similar to (8). 

(10) By (8) we have that $(y \, | \, f) \leq (y \, | \, e)$.
However, $(x \, | \, f),(y \, | \, f) \leq y$
and ${\bf d}(x \, | \, f) = {\bf d}(y \, | \, f)$.
Thus by axiom (OG3), we have that $(x \, | \, f) = (y \, | \, f)$.
Hence $(x \, | \, f) \leq (y \, | \, e)$.
\end{proof}

Let $G$ be an ordered groupoid and $H$ a subset of $G$. 
Then we say that $H$
is an {\em ordered subgroupoid} 
if it is a subgroupoid of $G$
and an ordered groupoid with respect to the induced order.
This is equivalent to the condition that $H$ be a subgroupoid of $G$
and that if $x \in H$ and $e \in H_{o}$  and $e \leq {\bf d}(x)$
then $(x \, | \, e) \in H$.

Let $\theta \colon \: G \rightarrow K$ 
be an injective ordered functor. 
The image of $\theta$ is a subgroupoid of $K$,
because if $\theta (x)\theta(y)$ is defined in $K$
then $\theta ({\bf d}(x)) = \theta ({\bf r}(y))$
and so $xy$ is defined in $G$;
this gives $\theta (x)\theta (y) = \theta (xy)$.
However, the image of $\theta$ need not be an ordered subgroupoid of $K$.
A stronger notion than an injective ordered functor is what we
term an {\em ordered embedding};
this is an ordered functor $\theta \colon \:G \rightarrow K$
such that for all $g,h \in G$
$$g \leq h \Leftrightarrow \theta (g) \leq \theta (h).$$
The image of $\theta$ is an ordered subgroupoid of $K$
which is isomorphic to $G$.

In verifying that a structure is an ordered groupoid, it is
sometimes more convenient to use the following characterization.
We shall need the following two axioms.
Let $G$ be a groupoid and $\leq$ a partial order defined on $G$.
The axioms (OI) and (OG4) are defined as follows:

\begin{description}

\item[{\rm (OI)}] $G_{o}$ is an order ideal of $G$.

\item[{\rm (OG4)}] For all $x \in G$ and $e \in G_{o}$,
if $e \leq {\bf d}(x)$ then there exists $y \in G$ 
such that $y \leq x$ and ${\bf d}(y) = e$.

\end{description}

\begin{lemma} Let $(G,\cdot)$ be a groupoid and $\leq$
a partial order defined on $G$.
Then $(G,\cdot,\leq)$ is an ordered groupoid if, and only if,
the axioms {\em (OG1)}, {\em (OG2)}, {\em (OI)} and {\em (OG4)} hold. 
\end{lemma}
\begin{proof}

If $G$ is an ordered groupoid 
then axioms (OG1) and (OG2) hold by definition, 
axiom (OI) holds by Lemma~2.4(7) and axiom (OG3) 
implies that axiom (OG4) holds.

To prove the converse, it is enough to show that axiom (OG3) holds 
because axiom (OG3)$^{\ast}$ follows from the other axioms for an ordered
groupoid by Lemma~2.4(6).
Let $u,v \leq x$ be such that ${\bf d}(u) = {\bf d}(v) = e$.
We shall show that $u = v$ which, together with axiom (OG4), will imply
that axiom (OG3) holds.
Clearly ${\bf d}(u) = {\bf r}(v^{-1}) = e$.
Thus $uv^{-1}$ is defined.
By  axiom (OG1),  we have that $v^{-1} \leq x^{-1}$.
Thus by axiom (OG2), we have that $uv^{-1} \leq xx^{-1}$.
Now $xx^{-1}$ is an identity and so by axiom (OI), the element $uv^{-1}$
is an identity.
Thus $u = v$, as required.
\end{proof}

\subsection{The Ehresmann-Schein-Nambooripad theorem}

From our results in Chapter~1 and by Lemma~2.1, we have the following.

\begin{proposition} Let $S$ be an inverse semigroup.  
Then $(S,\cdot,\leq )$ is an inductive groupoid.
\end{proposition}

The inductive groupoid associated with $S$ is denoted by ${\bf G}(S)$.

We now show how to construct an inverse semigroup from an inductive groupoid.

Let $G$ be an ordered groupoid and let 
$x,y \in  G$ be such that 
$e = {\bf d}(x) \wedge  {\bf r}(y)$ exists.  
Put 
$$x \otimes  y = (x\, | \, e)(e \, | \, y),$$ 
and call $x \otimes  y$ the 
{\it pseudoproduct}
of $x$ and $y$. 
It is immediate from the definition that the pseudoproduct is 
everywhere defined in an inductive groupoid.
The next result provides a neat, order-theoretic way of viewing the 
pseudoproduct.

\begin{lemma} Let $G$ be an ordered groupoid.  
For each pair $x,y \in  G$ put 
$$\langle x,y \rangle\, 
= \{(x',y') \in  G \times  G \colon 
{\bf d}(x' ) = {\bf r}(y') \mbox{ and } x' \leq   x 
\mbox{ and } y'  \leq  y\},$$
regarded as a subset of the ordered set $G \times G$.
Then $x \otimes  y$ exists if, and only if, there is a maximum element 
$(x',y')$ of $\langle x,y \rangle$.  
In which case, $x \otimes  y = x'y'$.
\end{lemma}
\begin{proof}
Suppose that $x \otimes y$ exists.
Then $e = {\bf d}(x) \wedge  {\bf r}(y)$ exists,
and $$((x \,|\, e),(e \,|\, y)) \in \langle x,y \rangle.$$
Let \mbox{$(u,v) \in \,\langle x,y \rangle$.}
Then 
$$u \leq x,\, v \leq y \mbox{ and } {\bf d}(u) = {\bf r}(v) = f,$$ 
say.
Thus by axioms (OG3) and (OG3)$^{\ast}$ we have that
$u = (x \, | \, f)$ and $v = (f \, | \, y)$.
By Lemma~2.4(1), 
${\bf d}(u) \leq {\bf d}(x)$ and ${\bf r}(v) \leq {\bf r}(y)$.
Thus $f \leq {\bf d}(x),{\bf r}(y)$,
and so, by assumption, $f \leq e$.
By Lemma~2.4(8),(9),  it follows that
$$u = (x \, | \, f) \leq (x \, | \, e)
\mbox{ and } 
v = (f \, | \, y) \leq (e \, | \, y).$$
Thus $((x \, | \, e),(e \, | \, y))$ is the maximum 
element of $\langle x,y \rangle$.  

Conversely, suppose that the maximum element of $\langle x,y \rangle$ 
exists and equals $(x',y')$.  
Put $e = {\bf d}(x') = {\bf r}(y')$.  
Clearly, $e \leq  {\bf d}(x),{\bf r}(y)$.  
Now let $f$ be any identity such that $f \leq  {\bf d}(x),{\bf r}(y)$.  
Then 
$$(x\, | \, f) \leq  x, \, (f\, | \, y) \leq  y 
\mbox{ and }
{\bf d}(x \, | \, f) = f = {\bf r}(f \, | \, y).$$  
Thus $((x \, | \, f),(f \, | \, y)) \in \, \langle x,y \rangle$
and so $((x \, | \, f),(f \, | \, y)) \leq (x',y')$.  
Hence $f \leq  e$, which implies that $e = {\bf d}(x) \wedge {\bf r}(y)$.
Thus $x \otimes y$ exists.
The proof of the last assertion is now immediate.
\end{proof}

It will be an immediate consequence of the following result that 
the pseudoproduct on an inductive groupoid is associative.

\begin{lemma}
Let $G$ be an ordered groupoid.  
Then for all $x,y,z \in G$
if $x \otimes  (y \otimes  z)$ and 
$(x \otimes  y) \otimes  z$ both exist 
then they are equal.
\end{lemma}
\begin{proof}
Let $(x \otimes  y) \otimes  z = az'$  
where  $(a,z')$  is  the  maximum  element  of 
$\langle x \otimes  y,z \rangle$.  
Let $x \otimes  y = x'y'$ where  $(x',y')$  
is  the  maximum  element  of $\langle x,y \rangle$.  
Then 
$$a \leq  x \otimes  y, \, z'  \leq  z, \, x'  \leq  x 
\mbox{ and } y'  \leq  y.$$
By Lemma~2.4(5), $a \leq  x'y'$ implies that there 
are elements 
$x'' \leq  x'$ and $y'' \leq  y'$ 
such that $a = x''y''$.  
Thus
$$(x \otimes  y) \otimes  z 
= (x''y'')z' = x''(y''z').$$
Now, $y'' \leq  y' \leq  y$ and $z' \leq  z$, 
so that $(y'',z') \in ~ \langle y,z \rangle$.  
Thus $y''z' \leq  y \otimes  z$.  
Similarly, $(x'',y''z') \in  ~\langle x,y \otimes  z \rangle$ and so 
$x''(y''z') \leq  x \otimes  (y \otimes  z)$.  
Hence
$(x \otimes  y) \otimes  z \leq  x \otimes  (y \otimes  z)$.  
The reverse inequality follows by symmetry.
\end{proof}
         
If $(G,\cdot,\leq)$ is an inductive groupoid, then  $(G,\otimes)$ will be denoted by ${\bf S}(G)$.
We can now show how to construct an inverse semigroup from an 
inductive groupoid.

\begin{proposition} Let $(G,\cdot ,\leq )$ be an inductive groupoid. 
\begin{enumerate}

\item $(G,\otimes )$ is an inverse semigroup. 

\item ${\bf G}({\bf S}(G,\cdot ,\leq )) = (G,\cdot ,\leq )$. 

\item For any inverse semigroup $S$ we have that ${\bf S}({\bf G}(S)) = S$.

\end{enumerate}
\end{proposition}
\begin{proof}
(1) By Lemma~2.8, $(G,\otimes )$ is a semigroup.  
If $x,y \in  G$ and $\exists x\cdot y$ in the groupoid $G$ 
then $x\cdot y = x \otimes  y$.  
But for each element $x \in  G$ we have that 
$x = x\cdot x^{-1}\cdot x$ and $x^{-1} = x^{-1}\cdot x\cdot x^{-1}$.  
Thus $(G,\otimes )$ is a regular semigroup.  
It is easy to check that the idempotents of $(G,\otimes )$ 
are precisely the identities of $(G,\cdot )$.  
Let $e$ and $f$ be two idempotents of $(G,\otimes )$.  
Then
$$e \otimes  f = (e \,|\, e \wedge f)(e \wedge f \,| \,f ) 
= e \wedge f 
= (f \,|\, e \wedge f) (e \wedge  f \,|\, e) = f \otimes e$$
so that the idempotents commute.
It follows that $(G,\otimes )$ is an inverse semigroup. 

(2) We show first that the natural partial order on $(G,\otimes )$ 
is just $\leq$.  
Suppose that $x = e \otimes  y$ in $(G,\otimes )$ for some idempotent $e$.  
Then $x = (e \wedge  {\bf r}(y)\, | \, y)$ and so 
$x \leq  y$ in $(G,\cdot ,\leq)$.  
Conversely, suppose that $x \leq  y$ in $(G,\cdot ,\leq)$.  
Then $x = ({\bf r}(x)\, | \, y)$.  
But $({\bf r}(x)\, | \, y) = {\bf r}(x) \otimes  y$.  
Thus $x \leq  y$ in $(G,\otimes )$.  
Now we turn to the restricted product.  
The restricted product of $x$ and $y$ is defined in $(G,\otimes )$ 
precisely when $x^{-1} \otimes  x = y \otimes  y^{-1}$. 
But from the properties of the pseudoproduct,
we have that $x^{-1} \otimes  x = x^{-1}\cdot x$
and $y \otimes  y^{-1} = y\cdot y^{-1}$ in $(G,\cdot)$.
Thus the restricted product of $x$ and $y$ exists in $(G,\otimes)$
precisely when the product $x\cdot y$ exists in $(G,\cdot )$.  
Thus ${\bf G}({\bf S}(G,\cdot ,\leq )) = (G,\cdot,\leq )$. 

(3) The pseudoproduct in ${\bf G}(S)$ is given by
$$s \otimes  t = (s\, | \, e)\cdot (e\, | \, t),$$
\noindent
where $e = {\bf d}(s) \wedge  {\bf r}(t)$ and the product on the 
right is the restricted product in $S$.  
But $(s\, | \, e) = se$ and $(e\, | \, t) = et$ and $e = s^{-1}stt^{-1}$ 
by Lemma~2.1.
Thus $s \otimes  t = st$.  
Hence ${\bf S}({\bf G}(S)) = S$.
\end{proof}

\begin{theorem}[Ehresmann-Schein-Nambooripad] The category of inverse semigroups and pre\-homo\-morph\-isms 
(respectively, homo\-morph\-isms) is isomorphic to the category of 
inductive groupoids and ordered functors (respectively, inductive functors).
\end{theorem} 
\begin{proof}

Define a function ${\bf G}$ from the category 
of inverse semigroups and prehomomorphisms 
to the category of inductive groupoids and ordered functors 
as follows:
for each inverse semigroup $S$ we define
${\bf G}(S) = (S,\cdot,\leq )$, an inductive groupoid by Proposition~2.6,
and if $\theta \colon \: S \rightarrow T$ is a prehomomorphism 
then ${\bf G}(\theta ) \colon {\bf G}(S) \rightarrow  {\bf G}(T)$ 
is defined to be the same function on the underlying sets;  
this is an ordered functor by Lemma~2.2.  
It is easy to check that ${\bf G}$ defines a functor.

Define a function ${\bf S}$ from the category of inductive
groupoids and ordered functors to the category of inverse
semigroups and prehomomorphisms as follows:
for each inductive groupoid $G$ we define
${\bf S}(G) = (G,\otimes )$, an inverse semigroup by Proposition~2.9,
and if $\phi \colon \:G \rightarrow H$ 
is an ordered functor between inductive groupoids 
then ${\bf S}(\phi ) \colon {\bf S}(G) \rightarrow  {\bf S}(H)$ 
is defined to be the same function on the underlying sets;  
this is a prehomomorphism by Lemma~2.2 
since it preserves the restricted product 
of ${\bf S}(G)$ and the natural partial order.  
It is easy to check that ${\bf S}$ is a functor.  

By Proposition~2.9, we have that 
${\bf G}({\bf S}(G,\cdot,\leq )) 
= (G,\cdot,\leq )$ and ${\bf S}({\bf G}(S)) = S$.  
It is now immediate that the category of inverse semigroups and
prehomomorphisms is isomorphic to the category of inductive
groupoids and ordered functors.
By Lemma~2.2 this isomorphism restricts to an  
isomorphism between the category of inverse semigroups and
semigroup homomorphisms and the category of inductive groupoids
and inductive functors.
\end{proof}

The above theorem can be viewed as a wide-ranging generalization of the result that a commutative idempotent semigroup can also be regarded as a meet semilattice.
One of its uses is in proving that a structure suspected of being an inverse semigroup actually is.

As usual in semigroup theory we have the slight annoyance of having to deal with the case of semigroups with and without zero separately.
An ordered groupoid is said to be {\em $\ast$-inductive}
if the following condition holds for each pair of identities:
if they have a lower bound, they have a greatest lower bound.
A $\ast$-inductive groupoid gives rise to an inverse semigroup with zero $(G^{0},\otimes)$:
adjoin a zero to the set $G$, and extend the pseudoproduct on $G$ to $G^{0}$ in such a way
that if $s,t \in G$ and $s \otimes t$ is not defined then put $s \otimes t = 0$,
and define all products with $0$ to be $0$.
Every inverse semigroup with zero arises in this way.

\subsection{Applications}

Ordered groupoids can be viewed as wide-ranging generalizations of inverse semigroups.
`Wide-ranging' because both groupoids and partially ordered sets are examples of ordered groupoids.
The category of ordered groupoids provides much more space for working with inverse semigroups.
In particular, constructions that lead out of the category of inverse semigroups may actually be possible in the larger category of ordered groupoids.
In this section, I shall illustrate this idea.

Let $G$ be an ordered subgroupoid of the ordered groupoid $H$.  
We say that $H$ is an {\em enlargement}
of $G$ if the following
three axioms hold:

\begin{description}

\item[{\rm (GE1)}] $G_{o}$ is an order ideal of $H_{o}$.

\item[{\rm (GE2)}] If $x \in H$ and ${\bf d}(x),{\bf r}(x) \in G$ 
then $x \in G$.

\item[{\rm (GE3)}] If $e \in H_{o}$ then there exists $x \in H$
such that ${\bf r}(x) = e$ and ${\bf d}(x) \in G$.

\end{description}

The following is proved in Section~8.3 of \cite{Lawson98}.

\begin{theorem}[The maximum enlargement theorem] \mbox{}

\begin{enumerate}

\item  Let $p \colon \: H \rightarrow K$ be an ordered, 
star injective functor between ordered groupoids.
Then there is an ordered groupoid $G$, 
an ordered embedding $i \colon \: H \rightarrow G$,
and an ordered covering functor $p'\colon \: G \rightarrow K$
such that $p'i = p$ where $G$ is an enlargement of $i(H)$.

\item Let $j \colon \: H \rightarrow G'$ be any ordered embedding
and let $p'' \colon \: G' \rightarrow K$ be an ordered covering
functor such that $p''j = p$.
Then there is a unique ordered functor $\theta \colon \: G \rightarrow G'$
such that $\theta i = j$ and $p'' \theta = p'$.

\end{enumerate}
\end{theorem}

The key point of the above theorem is that every star injective ordered functor can be factorized as an enlargement followed by a star bijective ordered functor,
and that this can be done in essentially one way.
Star bijective or covering functors have pleasant properties,
and if an ordered groupoid is an enlargement of another ordered groupoid then it is similar to it in structure.

We shall apply the above theorem to determine the structure of the $E$-unitary inverse semigroups introduced in Chapter~1.
It hinges on two observations.
First, in Theorem~2.26, we proved that an inverse semigroup is $E$-unitary if and only if the natural homomorphism
to its maximum group image is star injective;
second, in Theorem~3.6 an inverse semigroup is a semidirect product of a semilattice by a group if and only if that selfsame natural map is star bijective.

We begin by defining semidirect products of partially ordered sets by groups.
Let $G$ be a group and $X$ a partially ordered set.
We suppose that $G$ acts on $X$ by order automorphisms on the left.
Define a partial multiplication on $X \times G$ by
$(x,g)(y,h) = (x,gh)$ if $x = g \cdot y$ and undefined otherwise
and define a partial order on $X \times G$ by
$$(x,g) \leq (y,h) \Leftrightarrow x \leq y \mbox{ and } g = h.$$
The set $X \times G$ equipped with this partial multiplication
and partial order is denoted by $P(G,X)$.
The following is proved as Theorem~1 of Section~8.1 of \cite{Lawson98}.

\begin{lemma} $P(G,X)$ is an ordered groupoid, 
and the function $\pi_{2} \colon \: \: P(G,X) \rightarrow G$ 
defined by $\pi_{2}(x,g) = g$ is a surjective, ordered covering functor.
\end{lemma}

We say that the ordered groupoid 
$P(G,X)$ is a {\em semidirect product of a partially ordered set by a group}.
Such semidirect product ordered groupoids can be characterized abstractly.
The proof of the following is Theorem~3 of Section~8.1 of \cite{Lawson98}.

\begin{proposition} 
Let $\pi \colon \: \Pi \rightarrow G$ 
be an ordered covering functor from the ordered groupoid $\Pi$
onto a group $G$.  
Then $G$ acts on the poset $X = \Pi_{o}$
by order automorphisms, and there is an isomorphism of ordered groupoids
$\theta \colon \: \Pi \rightarrow P(G,X)$ 
such that $\pi_{2}\theta = \pi$.
\end{proposition}

The key point of Lemma~2.12 and Proposition~2.13 is that we can recognize when an ordered groupoid
is a semidirect product of a partially ordererd set by a group by checking to see if it admits
an ordered covering functor onto as group.
This is the ordered groupoid version of Theorem~3.6 of Chapter~1.

Now let $S$ be an $E$-unitary inverse semigroup.
Then the natural map from $S$ to its maximum group image is star injective.
This can be factorized using the maximum enlargment theorem into an enlargement followed by a covering map to a group.
It follows that every $E$-unitary inverse semigroup has an enlargement which is a semidirect product of a partially ordered set and a group.
For the full proof of the following see Theorem~4 of Section~8.1 of \cite{Lawson98}.

\begin{theorem} Let $S$ be an inverse semigroup with associated 
inductive groupoid ${\bf G}(S)$.  
Then $S$ is $E$-unitary if, and only if, 
there is an ordered embedding 
$\iota \colon \: {\bf G}(S) \rightarrow P(G,X)$
into some semidirect product of a poset by a group
such that $P(G,X)$ is an enlargement of $\iota ({\bf G}(S))$
in  such a way that the function
$\pi_{2} \colon \: P(G,X) \rightarrow G$ restricted to $\iota({\bf G}(S))$
is surjective.
\end{theorem}

The above result characterizes $E$-unitary inverse semigroups in terms of semidirect products,
but not of semilattices by groups as might be expected but of partially ordered sets by groups.
The disadvantage of the above result is that the structure of $E$-unitary inverse semigroups is described using ordered groupoids.
However, the information contained in the theorem can be couched in purely semigroup-theoretic language.
To do this we need the following classical definition.

Let $G$ be a group and $X$ a partially ordered set.
We shall suppose that $G$ acts on $X$ on the left by order automorphisms.
We denote the action of $g \in G$ on $x \in X$ by $g \cdot x$.
Let $Y$ be a subset of $X$ partially ordered by the induced ordering.
We say that $(G,X,Y)$ is a {\em McAlister triple} if the following three axioms hold:

\begin{description}

\item[{\rm (MT1)}] $Y$ is an order ideal of $X$ and a meet 
semilattice under the induced ordering.

\item[{\rm (MT2)}] $G \cdot Y = X$.

\item[{\rm (MT3)}] $g \cdot Y \cap Y \neq \emptyset$ for every $g \in G$.

\end{description}

Let $(G,X,Y)$ be a McAlister triple.
Put
$$P(G,X,Y) = \{(y,g) \in Y \times G \colon \: g^{-1} \cdot y \in Y \}.$$

\begin{lemma} Let $(e,g),(f,h) \in P(G,X,Y)$.
Then $e \wedge g \cdot f$ exists in the partially
ordered set $X$ and 
$(e \wedge g \cdot f,gh) \in P(G,X,Y)$.
\end{lemma}
\begin{proof}
By assumption, $g^{-1} \cdot e \in Y$.
Thus $g^{-1} \cdot e \wedge f$ exists since $Y$ is a semilattice.
Put $i = g^{-1} \cdot e \wedge f$.
Then $i \leq g^{-1} \cdot e$ and $i \leq f$.
Thus $g \cdot i \leq e$ and $g \cdot i \leq g \cdot f$.
Now let $j \leq e, g \cdot f$.
Then $g^{-1} \cdot j \leq g^{-1} \cdot e$ and $g^{-1} \cdot j \leq f$.
Hence $g^{-1} \cdot j \leq i$,
and so $j \leq g \cdot i$.
We have therefore shown that the meet $e \wedge g \cdot f$ exists.

To show that  
$(e \wedge g \cdot f,gh) \in P(G,X,Y)$, 
we have to show that the element
$(gh)^{-1} \cdot (e \wedge g \cdot f)$ 
belongs to $Y$.
Now
$$(gh)^{-1} \cdot (e \wedge g \cdot f)
=
h^{-1} \cdot (g^{-1} \cdot (e \wedge g \cdot f))$$ 
and
$$h^{-1} \cdot (g^{-1} \cdot (e \wedge g \cdot f)) 
\leq h^{-1} \cdot (g^{-1} \cdot (g \cdot f)) 
\leq h^{-1} \cdot f \in Y.$$ 
But $Y$ is an order ideal of $X$ and so
$(gh)^{-1} \cdot (e \wedge g \cdot f) \in Y$. 
\end{proof}

Define a product on $P(G,X,Y)$ by
$$(e,g)(f,h) = (e \wedge g \cdot f,gh).$$
It is well-defined by the above lemma.

A proof of the following is Theorem~9 of Section~7.2 of \cite{Lawson98}.

\begin{proposition} $P(G,X,Y)$ is an $E$-unitary inverse semigroup,
with semilattice of idempotents isomorphic to $Y$ and maximum
group homomorphic image isomorphic to $G$.
\end{proposition}

If we combine Theorem~2.4 with Proposition~2.16, we get the following first proved by Don McAlister.

\begin{theorem}[The $P$-theorem]
An inverse semigroup is $E$-unitary if and only if it is isomorphic to a $P$-semigroup.
\end{theorem}

\section{Ordered groupoids and left/right cancellative categories}

This section is a bridge between the inductive groupoid approach to studying inverse semigroups described in Section~2 and the
category action approach that we describe in Section~4.
Our goal is to show the extent to which ordered groupoids are related to left (respectively, right) cancellative categories.
The motivation for this section comes from the category-theoretic definition of a subobject.

\subsection{From left cancellative categories to ordered groupoids }

A {\em left cancellative} category is a category in which $xy = xz$ implies $y = z$.
It is therefore precisely a category of monomorphisms.
We define a {\em right cancellative} category dually, and a {\em cancellative} category is one which is both left and right cancellative.
A left (respectively, right) cancellative category with one identity is a left (respectively, right) cancellative monoid.
If $C$ is a subcategory of $D$, we say that it is {\em isomorphism-dense}
in $D$ if for each identity $e \in D_{o}$ there exists an identity
$f \in C_{o}$ and an isomorphism $x \in D$ such that 
$e \stackrel{x}{\longrightarrow} f$.
If $C$ is a subcategory of $D$, we say it is {\em full}
if $x \in D$ such that $\dom (x), \ran (x) \in C$ then $x \in C$.
A functor $F \colon C \rightarrow D$ is an {\em equivalence} if it is full, faithful and essentially surjective;
the first two conditions mean that the restriction $F \colon \mbox{hom}(e,f) \rightarrow \mbox{hom}(F(e),F(f))$
is surjective and injective respectively for all identities $e,f \in C_{o}$,
and the last condition means that each identity in $D$ is isomorphic to the image of an identity under $F$.

We now define two categories of structures, the relationship between them being the subject of this section.

\begin{description}
\item[$\mathcal{LC}$] The category of left cancellative categories
and their functors.
\item[$\mathcal{OG}$] The category of ordered groupoids and ordered functors.
\end{description}
We shall construct functors 
$$\G \colon \: \mathcal{LC} \rightarrow \mathcal{OG}
\text{ and }
\Le \colon \: \mathcal{OG} \rightarrow \mathcal{LC},$$
and describe their composites $\G \Le$ and $\Le \G$.

Let $C$ be a left cancellative category.
We shall construct an ordered groupoid $\G (C)$ from $C$.
Put 
$$U = \{(a,b) \in C \times C \colon \: \s (a) = \s (b)\}.$$
Define a relation $\sim$ on $U$ as follows:
$$(a,b) \sim (a',b') \Leftrightarrow (a,b) = (a',b')u \mbox{ for some isomorphism } u$$
where $(a',b')u = (a'u,b'u)$.
Then $\sim$ is an equivalence relation on $U$. 
Denote the equivalence class containing $(a,b)$ by $[a,b]$,
and the set of equivalence classes by ${\bf G}(C)$.
Define
$$\dom [a,b] = [b,b], \quad \ran [a,b] = [a,a] 
\mbox{ and }
[a,b]^{-1} = [b,a].$$
These are all well-defined.
Suppose $\dom [a,b] = \ran [c,d]$.
Then there exists an isomorphism $u$ in $C$ such that $b = cu$.
In this case, define the partial product
$$[a,b] \cdot [c,d] = [a,du].$$
The picture below illustrates why:
$$
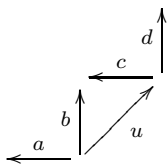
\diagram
&
&
&
\\
&
&
&\uto^{d} \lto_{c}
\\
&
&\lto_{a} \uto^{b} \urto_{u}
&
\enddiagram$$
We shall usually denote the partial product by concatenation.
Define a relation $\leq$ on ${\bf G}(C)$
by 
$$[a,b] \leq [c,d] \Leftrightarrow (a,b) = (c,d)p \mbox{ for some }p \in C.$$
This is well-defined and is a partial order.

If $\theta \colon \: C \rightarrow D$ is a functor
between two left cancellative categories, 
define the function 
${\bf G}(\theta) \colon \: {\bf G}(C) \rightarrow {\bf G}(D)$
by ${\bf G}(\theta)([a,b]) = [\theta (a),\theta (b)]$.

\begin{proposition} Let $C$ be a left cancellative category.
Then ${\bf G}(C)$ is an ordered groupoid,
and $\G$ is a functor from 
$\mathcal{LC}$ to $\mathcal{OG}$.
\end{proposition}
\begin{proof}
The set $(\G (C),\cdot)$ equipped with the partial binary operation defined above
is a groupoid in which the set of identities is
$$\G (C)_{o} = \{[a,a] \colon \: a \in C \}.$$ 
To show that $(\G (C), \cdot,\leq)$ is an ordered groupoid,
we verify that the axioms (OG1)--(OG3) hold.
It is immediate that (OG1) holds, 
and the proof of (OG2) is easy by left cancellativity.
We now prove that both (OG3) and (OG3)$^{\ast}$ hold.

Let $[a,a] \leq \dom [x,y]$.
Then $a = yp$ for some $p \in C$.
Define
$$([x,y]\,|\,[a,a]) = [xp,a].$$
It is easy to check that this is a well-defined restriction
whose uniqueness is a consequence of left cancellativity.
Let $[b,b] \leq \ran [x,y]$.
Then $b = xq$ for some $q \in C$.
Define
$$([b,b]\,|\,[x,y]) = [b,yq].$$
It is easy to check that this is a well-defined corestriction,
which is likewise unique.
It follows that $\G (C)$ is an ordered groupoid.

If $\theta \colon \: C \rightarrow D$ is a functor
between left cancellative categories,
then 
$\G (\theta)$ is an ordered functor from
$\G (C)$ to $\G (D)$. 
The proof that $\G$ is a functor is also straightforward.
\end{proof}

It is useful to know when the pseudoproduct of two 
elements of $\G(C)$ exists.

\begin{lemma} Let $C$ be a left cancellative category, and $\G(C)$ its associated ordered groupoid.
\begin{enumerate}

\item Let $[x,y],[w,z] \in \G(C)$.
Then $[x,y] \otimes [w,z]$ is defined 
if and only if
$y$ and $w$ have a pullback in $C$.

\item Let $C$ have the following additional property:
if $a$ and $b$ are any elements of $C$ such that $\ran (a) = \ran (b)$
and that can be completed to a commutative square $aa' = bb'$ for some elements
$a'$ and $b'$, then $a$ and $b$ have a pullback.
Then $\G (C)$ is $\ast$-inductive.
 
\end{enumerate}
\end{lemma}
\begin{proof}
(1) Consider the identities $[y,y]$ and $[w,w]$.
They have a lower bound
iff elements $p,q \in C$ can be found such that $yp = wq$;
and they have a greatest lower bound iff $y$ and $w$ have
a pullback in $C$; 
left cancellativity is once again used crucially.
Let $a = yp = wq$ be a pullback of $y$ and $w$. 
Then $[y,y] \wedge [w,w] = [a,a]$.
A simple calculation using the forms of the restriction
and corestriction yields
$$[x,y] \otimes [w,z]
=
[xp,zq].$$
In diagrammatic terms this is just:
$$\diagram
&
&
&
\\
&
&
&\uto^{z} \lto_{w}
\\
&
&\lto_{x} \uto^{y} 
&\lto_{p} \uto_{q}
\enddiagram$$

(2) This follows immediately from the proof of (1).
\end{proof}

Ordered groupoids of the form $\G (C)$ have extra properties.
Let $G$ be an ordered groupoid.
We say that $G$ has {\em maximal identities} if there is
a function $G_{o} \rightarrow G_{o}$, denoted by $e \mapsto e^{\circ}$,
that has the following two properties:
\begin{description}
\item[{\rm (MI1)}] $e \leq e^{\circ}$.
\item[{\rm (MI2)}] If $e \leq i^{\circ},j^{\circ}$ 
then $i^{\circ} = j^{\circ}$.
\end{description}
Observe that $e^{\circ}$ really is a maximal identity,
for if $e^{\circ} \leq f$, where $f$ is an identity, 
then $e^{\circ} \leq f \leq f^{\circ}$ by (MI1).
But then by (MI2), we have that $e^{\circ} = f^{\circ}$, 
and so $e^{\circ} = f$.

We define the relation $\D$ on an ordered groupoid $G$ by
$g \,\D\, h$ iff they are in the same connected component
of the groupoid $G$.
It is immediate that $\D$ is an equivalence relation on $G$.

\begin{proposition} Let $C$ be a left cancellative category.
Then $\G (C)$ is an ordered groupoid with maximal identities.
In addition, each $\D$-class contains a maximal identity.
\end{proposition}
\begin{proof}
For each identity $[a,a] \in \G (C)$ define
$$[a,a]^{\circ} = [\f (a),\f (a)].$$
This is evidently a function from $\G (C)_{o}$ to $\G (C)_{o}$.
Because $(a,a) = (\f (a),\f (a))a$ we have that
$[a,a] \leq [a,a]^{\circ}$.
Thus (MI1) holds.
Suppose that 
$$[a,a] \leq [\ran (b), \ran (b)],[\ran (c),\ran (c)].$$
Then $a = \ran (b)p = \ran (c)q$ for some $p$ and $q$.
But then $a = p = q$ and so $\ran (b) = \ran (p) = \ran (a)$,
and $\ran (c) = \ran (p)$.
Thus $[\ran (b), \ran (b)] = [\ran (c),\ran (c)]$,
and so (MI2) holds.
It follows that $\G (C)$ has maximal identities.
Observe that an identity is of the form $[a,a]^{\circ}$
iff it is of the form $[e,e]$ for some identity $e \in C_{o}$.

Let $[a,a]$ be an arbitrary identity.
Consider the maximal identity 
$[\s (a),\s (a)]$.
Then $[a, \s (a)] \in \G (C)$ is
such that
$\dom [a, \s (a)] = [\s (a), \s (a)]$ 
and
$\ran [a, \s (a)] = [a,a]$.
Hence $[a,a] \, \D \, [\s (a), \s (a)]$.
It follows that each $\D$-class contains a maximal identity.
\end{proof}

We have shown how to construct an ordered groupoid $\G (C)$ from a left cancellative category.
It is likewise possible to construct an ordered groupoid $\G' (C)$ from a {\em right} cancellative category $C$.
Just as the construction of $\G (C)$ is a generalization of the construction of subobjects,
so the construction of $\G' (C)$ is a generalization of quotient objects.
The set $U'$ is defined to consist of those pairs of elements $(a,b)$
of $C$ such that ${\bf r}(a) = {\bf r}(b)$.
The equivalence relation $\sim$ is defined by
$$(a,b) \sim (a',b') \Leftrightarrow (a,b) = u(a',b')$$
for some isomorphism $u$. 
We define $\G' (C)$ to consist of $\sim$-equivalence classes.
The following is immediate

\begin{proposition} Let $C$ be a right cancellative category.
Then $\G' (C)$ is an ordered groupoid with maximal identities.
In addition, each $\mathcal{D}$-class contains a maximal identity.
\end{proposition}

\subsection{From ordered groupoids to left cancellative categories}

Let $G$ be an ordered groupoid.
Define
$$\Le (G) = \{(e,x) \in G_{o} \times G \colon \: {\bf r}(x) \leq e\}$$
and
$$\s (e,x) = (\dom (x),\dom (x))
\mbox{ and }
\f (e,x) = (e,e)$$
and define a partial product on $\Le (G)$ as follows:
if $\s (e,x) = \f (f,y)$ then
$$(e,x)(f,y) = (e,x \otimes y),$$
otherwise it is undefined.
If $\theta \colon \: G \rightarrow H$ is an ordered functor
between two ordered groupoids, define 
${\bf L}(\theta) \colon \: {\bf L}(G) \rightarrow {\bf L}(H)$
by ${\bf L}(\theta)(e,x) = (\theta (e),\theta (x))$.
Since $\ran (x) \leq e$ and $\theta$ is an order-preserving functor,
we have that $\theta (\ran (x)) = \ran ( \theta (x))$
and $\ran (\theta (x)) \leq \theta (e)$.   
Thus $\Le (\theta)(e,x)$ is an element of $\Le (G)$,
and so $\Le (\theta)$ is a well-defined function.

\begin{proposition} Let $G$ be an ordered groupoid.
Then ${\bf L}(G)$ is a left cancellative category,
and
$\Le$ defines a functor from $\mathcal{OG}$ to $\mathcal{LC}$.
The identities of ${\bf L}(G)$ are those elements of the form
$(e,e)$ where $e$ is an identity of $G$.
The invertible elements of $\Le (G)$ are those of the form $(\ran (x),x)$
and constitute a groupoid isomorphic to $G$.
\end{proposition}
\begin{proof}
We regard the element $(e,x)$ in diagrammatic terms
as follows
$$
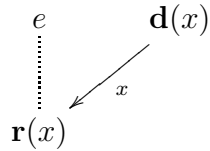
\diagram
&e \ddotted
&\dom (x)\dlto^{x}
\\
&\ran (x)
&
\enddiagram$$
The dotted line indicates the natural partial order.
The fact that $\s (e,x) = \f (f,y)$
means, in diagrammatic terms, that we have
$$
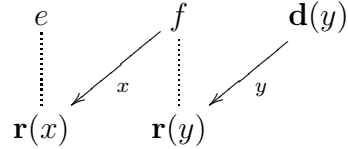
\diagram
&e \ddotted
&f \dlto^{x} \ddotted
&\dom (y) \dlto^{y}
\\
&\ran (x)
&\ran (y)
&
\enddiagram$$
Thus using the restriction we can construct the following
$$
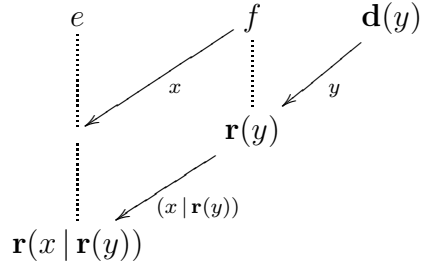
\diagram
&e \ddotted
&f \dlto^{x} \ddotted
&\dom (y) \dlto^{y}
\\
&\ddotted
&\ran (y)  \dlto^{(x\,|\,\ran (y))}
&
\\
&\ran (x \,|\, \ran (y))
&
&
\enddiagram$$
But $(x \,|\, \ran (y))y = x \otimes y$.
It is now evident that $\Le (G)$ is a category
whose set of identities is
$$\Le (G)_{o} = \{(e,e) \colon \: e \in C_{o} \}.$$
We prove that it is left cancellative.
Suppose that 
$$(e,x)(f,y) = (e,x)(f,z).$$
Then $(x \,|\, \ran (y))y = (x \,|\, \ran (z))z$.
But $(x \,|\, \ran (y)),(x \,|\, \ran (z)) \leq x$
and $\ran (x \,|\, \ran (y)) = \ran (x \,|\, \ran (z))$.
Hence $(x \,|\, \ran (y)) = (x \,|\, \ran (z))$ and so $y = z$.
We have therefore proved that $(f,y) = (f,z)$.

The proof that $\Le (\theta)$ is a functor
follows from the fact that if
$\theta \colon \: G \rightarrow H$ is an ordered functor and $e \leq \dom (x)$,
then $\theta (e) \leq \dom (\theta (x))$
and so $\theta (x \,|\, e) = (\theta (x) \,|\, \theta (e))$.
Consequently $\Le$ is a functor.

It is easy to check that the inverse of $(\ran (x),x)$ is $(\ran (x^{-1}), x^{-1})$,
and that only elements of the form $(\ran (x),x)$ are invertible.
The function from $G$ to $\Le (G)$ defined by $x \mapsto (\ran (x),x)$ 
induces an isomorphism between $G$ and the groupoid of invertible elements of $\Le (G)$.
\end{proof}

We may also construct a {\em right} cancellative category
from an ordered groupoid.
Let
$$\R (G) = \{(x,e) \in G \times G_{o} \colon \: {\bf d}(x) \leq e\}.$$
Define 
$${\bf d}(x,e) = (e,e)
\mbox{ and }
{\bf r}(x,e) = ({\bf r}(x),{\bf r}(x))$$
and define a partial product on $\R (G)$ as follows:
if ${\bf d}(x,e) = {\bf r}(y,f)$ then
$(x,e)(y,f) = (x \otimes y,f)$,
else it is undefined.

\begin{proposition} Let $G$ be an ordered groupoid.
Then $\R (G)$ is a right cancellative category,
and $\Le (G)^{\scriptsize op}$ is isomorphic
to $\R (G)$.
\end{proposition}
\begin{proof}
It is easy to check that $\R (G)$ is a category.
We now prove that $\Le (G)^{\scriptsize op}$ is isomorphic
to $\R (G)$.
Define a function
$\theta \colon {\bf L}(G)^{\scriptsize op} \rightarrow {\bf R}(G)$
by $\theta (e,x)^{\scriptsize op} = (x^{-1},e)$.
It is immediate that this function is bijective
and maps identities to identities.
Suppose that
$(f,y)^{\scriptsize op}(e,x)^{\scriptsize op}$ 
is defined in ${\bf L}(G)^{\scriptsize op}$.
By definition
$$(f,y)^{\scriptsize op}(e,x)^{\scriptsize op}
= ((e,x)(f,y))^{\scriptsize op}$$
and so
${\bf d}(x) = f$.
By definition,
$$\theta (f,y)^{\scriptsize op} = (y^{-1},f)
\mbox{ and }
\theta (e,x)^{\scriptsize op} = (x^{-1},e).$$
It follows that 
$\theta (f,y)^{\scriptsize op}\theta (e,x)^{\scriptsize op}$
is defined in ${\bf R}(G)$.
We have that
$$\theta (f,y)^{\scriptsize op}\theta (e,x)^{\scriptsize op}
= 
(y^{-1},f)(x^{-1},e)
=
(y^{-1} \otimes x^{-1},e)
=
\theta (e,x \otimes y).$$
Thus $\theta$ is an isomorphism from ${\bf L}(G)^{\scriptsize op}$ 
to ${\bf R}(G)$.
\end{proof}

\subsection{Forward and back}
We shall now describe the relationship between $\Le \G (C)$ and $C$, and between $\G \Le (G)$ and $G$.
Our first theorem tells us that up to equivalence {\em every} left cancellative category can be constructed from an ordered groupoid.

\begin{theorem} Let $C$ be a left cancellative category.
Then $C' = {\bf L}{\bf G}(C)$ is a left cancellative category.
Define
$$\iota \colon \: C \rightarrow C'
\mbox{ by }
\iota (a) = ([\f (a),\f (a)],[a, \s(a)]).$$
Then $\iota$ is an injective functor which embeds 
$C$ in $C'$ as a full, isomorphism-dense subcategory.
In particular, $C$ and ${\bf L}{\bf G}(C)$
are equivalent categories.
\end{theorem}
\begin{proof}
Identities of $C'$ have the form
$([z,z],[z,z])$ where $z \in C$.
Thus $\iota$ maps identities to identities.
Suppose that $ab$ is defined in $C$.
Then $\iota (a)\iota (b)$ is defined in $C'$.
From the fact that 
$$[a,\s (a)] \otimes [b,\s (b)] = [ab,\s (b)]$$
we quickly deduce that
$\iota (ab) = \iota (a)\iota (b)$.
Thus $\iota$ is a functor.

Suppose that $\iota (a) = \iota (b)$.
Then there are isomorphisms $u,v \in C$ such that
$$\f (a) = \f (b)u, \quad a = bv \mbox{ and } \s (a) = \s (b)v.$$ 
Thus $v = \s (a)$ and so $a = b$.
Hence $\iota$ is injective.

Let $([c,c],[a,b])$ be an arbitrary element of $C'$ such that
$$\s ([c,c],[a,b]), \f ([c,c],[a,b]) \in \iota (C).$$
From $\ran [a,b] \leq [c,c]$,
we get that $a = cp$ for some $p \in C$.
By our assumption, we have that $[b,b] = [e,e]$ and $[c,c] = [f,f]$ for some
$e,f \in C_{o}$.
Thus there are isomorphisms $u,v$ such that
$b = eu$ and $c = fv$.
Now
$$[c,c] = [\f (v), \f (v)]
\mbox{ and }
[a,b] = [vpu^{-1},\f (u)].$$
But $\iota (vpu^{-1}) = ([\f (v),\ f (v)],[vpu^{-1},\f (p)])$.
Thus $\iota$ is a full functor.

Finally, let $([a,a],[a,a])$ be an arbitrary identity in $C'$.
Consider the element $([a,a],[a,\s (a)])$ of $C'$.
It is an isomorphism such that
$$\f ([a,a],[a,\s (a)]) = ([a,a],[a,a])
\mbox{ and } 
\s ([a,a],[a,\s (a)]) = \iota (\s (a)).$$
\end{proof}

We shall now describe the relationship between $G$ and $\G \Le (G)$.
This is not as satisfactory and explains why we cannot put ordered groupoids and left cancellative categories on a par.

The ordered groupoid $\G \Le (G)$ is a somewhat complicated object, so to help us see what is going on,  we construct an isomorphic copy that is easier to understand.

Let $G$ be an ordered groupoid.
Define $\overline{G}$ as follows:
$$\overline{G} = \{\langle e,x,f \rangle \in G_{o} \times G \times G_{o}  \colon \: \dom (x) \leq f, \ran (x) \leq e \}.$$
Define 
$$\dom (\langle e,x,f \rangle) = \, \langle f, \dom (x), f \rangle
\mbox{ and }
\ran (\langle e,x,f \rangle) = \, \langle e, \ran (x), e \rangle,$$
and
$$\langle e,x,f \rangle^{-1} = \, \langle f,x^{-1},e \rangle.$$
These are all elements of $\overline{G}$.
If $\dom \langle e,x,f \rangle = \ran \langle i,y,j \rangle$,
define
$$\langle e,x,f \rangle \langle i,y,j \rangle = \langle e,xy,j \rangle.$$
Then $\overline{G}$ is easily seen to be a groupoid in which the identities are the elements of the form $\langle e,f,e \rangle$.
Define
$$\langle e,x,f \rangle \, \leq \, \langle e',x',f' \rangle$$ 
iff $e = e'$, $f = f'$ and $x \leq x'$.
Finally, define 
$$\langle e,f,e \rangle^{\circ} = \langle e,e,e \rangle.$$
Observe that if $f \leq e$ where $e,f \in G_{o}$,
then $\dom (\langle e,f,f \rangle) = \langle f,f,f \rangle$
and $\ran (\langle e,f,f \rangle) = \langle e,f,e \rangle$.

\begin{lemma} Let $G$ be an ordered groupoid.
Then $\overline{G}$ is an ordered groupoid with maximal identities in which each $\mathcal{D}$-class contains a maximal identity.
The map $\kappa \colon \: \overline{G} \rightarrow G$ is an ordered functor that is order-reflecting.
\end{lemma}

The connection with $\G \Le (G)$ is described in the following result.

\begin{proposition} Let $G$ be an ordered groupoid.
Then $G' = \G \Le (G)$ is isomorphic to $\overline{G}$.
\end{proposition}
\begin{proof}
We show first that every element of $G'$
can be written uniquely in the form $[(e,w),(f,\dom (w))]$ for some $w \in G$.
Let $[(e,x),(f,y)] \in G'$.
Then $\dom (x) = \dom (y)$ and $\ran (x) \leq e$
and $\ran (y) \leq f$.
The pair $(\dom (x),y^{-1})$ is an isomorphism in $\L (G)$,
and
$$(e,x)(\dom (x),y^{-1}) = (e,xy^{-1})
\mbox{ and }
(f,y)(\dom (x),y^{-1}) = (f,\ran (y)).$$
Put $w = xy^{-1}$.
We have proved that
$$[(e,x),(f,y)] = [(e,w),(f, \dom (w))].$$
Suppose now that 
$$[(e,w),(f,\dom (w))] = [(e',w'),(f',\dom (w'))].$$
Then there is an invertible element
$(\dom (x),x)$ in $\Le (G)$ 
such that
$$(f ,\dom (w)) = (f',\dom(w'))(\dom (x),x)$$
where $\dom (w') = \dom (x)$.
It follows that $f = f'$ and $x = \dom (w)$.
Thus $(\dom (x),x)$ is an identity and so
$(e,w) = (e',w')$ and $(f,\dom (w)) = (f',\dom (w'))$.
It follows that $e = e'$, $f = f'$ and $w = w'$.

Define 
$$\alpha \colon \: \overline{G} \rightarrow G' 
\mbox{ by }
\alpha (\langle e,x,f \rangle) = [(e,x),(f ,\dom (x))].$$
Then the two results above show that $\alpha$ is a bijection.
We shall prove that it is an isomorphism of ordered groupoids.

Suppose that $\langle e,x,f \rangle \, \leq \, \langle e,x',f \rangle$.
Then $x \leq x'$.
The pair $(\dom (x'), \dom (x)) \in \Le (G)$
and
$$(e,x')(\dom (x'),\dom (x)) = (e,x)
\mbox{ and }
(f,\dom (x))(\dom (x'),\dom (x)) = (f, \dom (x)).$$
Thus $\alpha (\langle e,x,f \rangle) \, \leq \, \alpha (\langle e,x'f \rangle)$, and so
$\alpha$ is order-preserving.
Suppose that $\alpha (\langle e,x,f \rangle) \, \leq \, \alpha (\langle e',x',f' \rangle)$.
Then there is an element $(i,a) \in \Le (G)$ such that
$$(e,x) = (e',x')(i,a)
\mbox{ and }
(f, \dom (x)) = (f',\dom (x'))(i,a).$$
It readily follows from this that
$e = e'$, $f = f'$ and $x \leq x'$.
Hence $\langle e,x,f \rangle  \, \leq \, \langle e',x',f' \rangle$. 
We have therefore proved that $\alpha$ is an order isomorphism.

It remains to show that $\alpha$ is a functor.
It is easy to check that $\alpha$ maps identities to identities.
Suppose that $\langle f,x,e \rangle \langle e,y,i \rangle$ is defined in $\overline{G}$.
By definition
$$\alpha (\langle f,x,e \rangle ) = [(f,x),(e,\dom(x))]
\mbox{ and }
\alpha (\langle e,y,i \rangle) = [(e,y),(i,\dom (y))].$$
Observe that $(e,\dom (x)) = (e,y)(\dom(y),y^{-1})$
where $(\dom (y),y^{-1})$ is invertible.
It follows that
$\alpha (\langle f,x,e \rangle)\alpha (\langle e,y,i \rangle)$
is defined in $G'$
and is equal to
$[(f,x),(i,y^{-1})]$.
But
$$[(f,x),(i,y^{-1})] = [(f,xy),(i,\dom (xy))] = \alpha (\langle f,xy,i \rangle)$$
making use of the isomorphism $(\dom (x),y)$.
Thus $\alpha$ is a functor.
\end{proof}

To get a sharper connection betweeen $G$ and $\overline{G}$ we need to assume more about $G$.

\begin{proposition} Let $G$ be an ordered groupoid with maximal identities in which each $\mathcal{D}$-class contains
a maximal identity.
Then $\overline{G}$ is an enlargement of $G$.
\end{proposition}
\begin{proof} Define $\pi \colon G \rightarrow \overline{G}$ by $\pi (g) = (\ran (g)^{\circ}, g, \dom (g)^{\circ})$.
Then $\pi$ is an ordered embedding and $\overline{G}$ is an enlargement of the image of $\pi$.
\end{proof}

\subsection{Rooted categories}

The results we have obtained so far on ordered groupoids and left cancellative categories show that we almost have a correspondence between them.
But to get sharper results, we need to restrict the class of left cancellative categories we consider and correspondingly the class of ordered groupoids.
In both cases, we need an `anchor'.

A {\em weak terminal identity} in a category is an identity $1$ with the property that for each identity $e$ there is an arrow from $e$ to $1$.

\begin{lemma} Let $G$ be an ordered groupoid with maximum identity $1$.
Then $\Le (G)$ has a weak terminal identity $(1,1)$.
\end{lemma}
\begin{proof}
Let $(e,e) \in \Le (G)_{o}$.
Then $e \leq 1$ and so $(1,e) \in \Le (G)_{o}$.
\end{proof}

Let $C$ be a left cancellative category.
The set $[a,a] \otimes \G (C) \otimes [a,a]$
is the set of all products $[a,a] \otimes [x,y] \otimes [a,a]$
where they are defined and forms an ordered subgroupoid of $\G (C)$.
 
\begin{lemma}
Let $C$ be a left cancellative category with weak terminal identity $1$.
Then $[a,b]  \in [1,1] \otimes \G (C) \otimes [1,1]$
if and only if $\ran (a) = 1 = \ran (b)$.
\end{lemma}
\begin{proof}
Let $[a,b] \in [1,1] \otimes {\bf G}(C) \otimes [1,1]$.
Then $[a,a],[b,b] \leq [1,1]$.
It follows that there are $p,q \in C$
such that $a = 1p$ and $b = 1q$.
Thus, in particular, ${\bf r}(a) = {\bf r}(b) = 1$.
Conversely, suppose that
${\bf r}(a) = {\bf r}(b) = 1$.
Then $a = 1a$ and $b = 1b$ and so
$[a,a],[b,b] \leq [1,1]$.
Hence $[a,b] \in [1,1] \otimes {\bf G}(C) \otimes [1,1]$.
\end{proof}

Let $C$ be a left cancellative category with weak terminal identity $1$.
Put $\G_{1}(C) = [1,1] \otimes {\bf G}(C) \otimes [1,1]$.

\begin{proposition} Let $G$ be an ordered groupoid with maximum identity $1$.
Then $G$ is isomorphic to $\G_{1}(\Le (G))$.
\end{proposition}
\begin{proof}
From the proof of Proposition~3.10,
the function $\pi \colon \: G \rightarrow \bar{G}$
defined by
$$\pi (x) = \langle 1,x,1 \rangle$$
is a well-defined, injective ordered functor
since here $e^{\circ} = 1$ for each identity $e$ in $G$.
In addition, $G$ is isomorphic to $\pi (G)$.
Under the isomorphism $\alpha \colon \: \overline{G} \rightarrow \G \Le (G)$
of Proposition~3.9, 
we have that $\alpha \pi$ is contained in $\G_{1}\Le (G)$.
Let $[(1,x),(1,y)]$ be an arbitrary element of $\G_{1}\Le (G)$.
Then $(\dom (y),y^{-1})$ is an isomorphism
and $(1,x)(\dom (y),y^{-1}) = (1,xy^{-1})$
and $(1,y))(\dom (y),y^{-1}) = (1,yy^{-1})$.
Thus
$$[(1,x),(1,y)] = [(1,xy^{-1}),(1,yy^{-1})]
=
\alpha \pi (xy^{-1}).$$
It follows that $G$ is isomorphic to $\G_{1}(\Le (G))$.
\end{proof}

We now cast the above results into a more usable form.
A left cancellative category $C$ with a weak terminal identity 1 is called a {\em left rooted category}.
Put 
$$U = \{(a,b) \in C \times C \colon \: \dom (a) = \dom (b), \ran (a) = 1 = \ran (b)\}.$$
Observe that both $a$ and $b$ have codomain 1.
Define a relation $\sim$ on $U$ by
$$(a,b) \sim (a',b') \Leftrightarrow (a,b) = (a',b')u \mbox{ for some isomorphism } u$$
where $(a',b')u = (a'u,b'u)$.
Then $\sim$ is an equivalence relation on $U$. 
Denote the equivalence class containing $(a,b)$ by $[a,b]$,
and the set of equivalence classes by ${\bf G}^{l}(C)$.
We may think of $[a,b]$ as modelling a partial bijection with domain of definition described by $[b,b]$ and with range described by $[a,a]$
both of these being subobjects of 1.
Define
$$\dom [a,b] = [b,b], \quad \ran [a,b] = [a,a] 
\mbox{ and }
[a,b]^{-1} = [b,a].$$
If $\dom [a,b] = \ran [c,d]$.
Then there exists an isomorphism $u$ in $C$ such that $b = cu$.
In this case, define the partial product
$$[a,b] \cdot [c,d] = [a,du].$$
Define a relation $\leq$ on ${\bf G}^{l}(C)$
by 
$$[a,b] \leq [c,d] \Leftrightarrow (a,b) = (c,d)p \mbox{ for some }p \in C.$$
This is well-defined and is a partial order.
In this way, $\G^{l}(C)$ is an ordered groupoid with a maximum identity.
Let $C$ and $D$ be two left rooted categories with weak terminal identities $1_{C}$ and $1_{D}$ respectively.
Let $\theta \colon \: C \rightarrow D$ be an equivalence of categories such that $F(1_{C}) = 1_{D}$.
Then the function
${\bf G}^{l}(\theta) \colon \: {\bf G}^{l}(C) \rightarrow {\bf G}^{l}(D)$
defined by ${\bf G}^{l}(\theta)([a,b]) = [\theta (a),\theta (b)]$
is an isomorphism of ordered groupoids with maximum identities.

We shall say that a left rooted category has {\em all allowable pullbacks}
if whenever $a$ and $b$ are elements of $C$ such that $\ran (a) = \ran (b)$
and that can be completed to a commutative square $aa' = bb'$ for some elements $a'$ and $b'$, 
then $a$ and $b$ have a pullback.
It is worth noting that it is enough to assume this condition for those pairs $a$ and $b$ where in addition $\ran (a) = \ran (b)$.
In this case, $\G^{l} (C)$ is $\ast$-inductive and so when we adjoin a zero we get an inverse monoid with zero.
The ordered groupoid $\G^{l} (C)$ is inductive when $C$ has {\em all pullbacks}.

If $S$ is an inverse monoid with zero, 
we denote by $\mathcal{L}(S)$ the left rooted category of the ordered groupoid of non-zero elements of $S$,
its right rooted category, denoted by $\mathcal{R}(S)$, is defined similarly.
We therefore have the following theorem with which is associated an evident dual theorem referring to right Leech categories.

\begin{theorem} \mbox{}

\begin{enumerate}

\item Let $G$ be an ordered groupoid with a maximum identity.
Then $\Le (G)$ is a left rooted category and $\G^{l}(\Le (G))$ is isomorphic to $G$.

\item Let $C$ be a left rooted category.
Then $\G^{l}(C)$ is an ordered groupoid with a maximum identity
and $\Le (\G^{l}(C))$ is a left rooted category equivalent to $C$. 

\item  Each inverse monoid with zero is determined by a left cancellative category with a weak terminal identity
that has all allowable pullbacks.

\item  Each inverse monoid is determined by a left cancellative category with a weak terminal identity that has all pullbacks.

\end{enumerate}
\end{theorem}
\begin{proof} We prove (2).
For each identity $e$ in $C$ choose an arrow $c_{e} \colon e \rightarrow 1$.
We choose $c_{1} = 1$.
Let $f \stackrel{a}{\longleftarrow} e$ be an element of $C$.
Then $a$ and $c_{e}$ have the same domain.
Thus $c_{f}a$ and $c_{e}$ have both the same domain and range equal to 1.
Hence $[c_{f}a,c_{e}]$ is an element of $\G^{l}(C)$.
Also $[c_{f},c_{f}]$ is an identity of $\G^{l}(C)$. 
By construction $\ran [c_{f}a,c_{e}] \leq [c_{f},c_{f}]$.
Thus $([c_{f},c_{f}], [c_{f}a,c_{e}])$ is an element of $\Le (\G^{l}(C))$.
Define $\theta \colon C \rightarrow  \Le (\G^{l}(C))$
by 
$$\theta (a) =  ([c_{f},c_{f}], [c_{f}a,c_{e}]).$$
The function $\theta$ is a functor that maps the weak terminal identity $1$ to the weak terminal identity $([1,1],[1,1])$
and is full, faithful and essentially surjective.\end{proof}

There is now an obvious question which has a familiar answer.

\begin{proposition}
Let $S$ be an inverse monoid.
Then its associated left rooted category is cancellative if and only if $S$ is $E$-unitary.
\end{proposition}
\begin{proof}
Suppose that $\mathcal{L}(S)$ is cancellative.
Let $e \leq s$.
Then 
$$(1,s)(s^{-1}s,e) = (1,s^{-1}s)(s^{-1}s,e)$$ 
and so by right cancellation $s = s^{-1}s$, an idempotent as required.

Conversely, suppose that $S$ is $E$-unitary.
We prove that  $\mathcal{L}(S)$ is right cancellative.
Let $(e,s)(f,u) = (e,t)(f,u)$.
Then $su = tu$.
It follows that $s^{-1}suu^{-1} = s^{-1}tuu^{-1}$ and so $s^{-1}t$ is an idempotent
using the fact that $S$ is $E$-unitary.
Similarly $st^{-1}$ is an idempotent.
Hence $s$ and $t$ are compatible but also $s^{-1}s = t^{-1}t$.
Thus by Lemma~2.24, we have that $s = t$ as required.
\end{proof}

It is interesting to consider what happens in the case when the inverse monoid is $0$-bisimple because we can then replace categories by monoids.
Let $S$ be $0$-bisimple.
Put $L_{1} = \{s \in S \colon s^{-1}s = 1 \}$.
Then $L_{1}$ is a left cancellative monoid.
It is isomorphic to the local monoid at the identity $(1,1)$ in the left rooted category $\mathcal{L}(S)$
and under our assumption on $S$ is actually equivalent to it.
We therefore have the following. 

\begin{theorem} 
Each $0$-bisimple inverse monoid with zero is determined by a left cancellative monoid
with the property that any two principal right ideals are either disjoint or their intersection is again a principal right ideal.
\end{theorem}

\begin{remark}
{\em What we have proved in this section for inverse monoids can also be generalized to inverse categories}
\end{remark}

\section{Affine systems}\setcounter{theorem}{0}

Inverse monoids and, more generally, ordered groupoids with maximum identities can be described by means of suitable left cancellative categories 
equipped with weak terminal identities or by means of suitable right cancellative categories equipped with weak initial identities.
The problem now is describing arbitrary inverse semigroups or, more generally, arbitrary ordered groupoids.
The solution is similar to what happens when we want to replace vector spaces, which have a distinguished origin, with affine spaces which don't:
we have to work with actions.
In our case, and choosing sides, we work with (right cancellative) categories acting on the left on principal groupoids,
the groupoids arising from equivalence relations.
We shall see that our earlier description of inverse monoids or ordered groupoids with maximum identities is a special case.
In outline, we do the following
\begin{itemize}
\item A category $C$ acts on a principal groupoid $H$ on the left.
\item This action induces a preorder $\preceq$ on $H$ whose associated equivalence relation is $\equiv$.
\item The quotient structure $H/\equiv$ is a groupoid on which the preorder induces an order.
\item The groupoid $H/\equiv$ is ordered and every ordered groupoid is isomorphic to one constructed in this way.
\end{itemize}

\subsection{From category actions to  ordered groupoids}

Before I give the formal definition, I want to motivate it by reconsidering how we defined an ordered groupoid
with identity from a {\em right} cancellative category $C$ with a weak {\em  initial} identity 1.
Incidently, the shift from left to right rooted categories is simply a consequence of the fact that I shall work with {\em left} actions below.
Form the set of ordered pairs
$$U = \{(a,b) \in C \times C \colon \ran (a) = \ran (b), \dom (a) = 1 = \dom (b) \}.$$
We may define a partial action of $C$ on $U$ by defining $c \cdot (a,b) = (ca,cb)$ only when $\dom (c) = \ran (a)$.
If we define $\pi (a,b) = \ran (a)$ then we can rewrite the condition for the action of $c$ on $(a,b)$ to be defined by $\dom (a) = \pi (a,b)$. 
What is $U$?
It is the groupoid corresponding to an equivalence relation defined on the set $C1$:
namely $a$ and $b$ are related if and only if $\ran (a) = \ran (b)$.
Suppose that $(a,b) = c(a',b')$ and $(a',b') = d(a,b)$.
Then because $C$ is right cancellative, the elements $c$ and $d$ are mutually invertible.
It follows that $(a,b) \sim (a',b')$.
Thus the action of $C$ on $U$ can be used to construct the elements of the ordered groupoid $\G^{r}(C)$.
In order to define $U$ in this case we needed a weak initial identity,
but the same construction would go through if we started with the pair $(G,U)$.
By axiomatizing the properties of such pairs leads to the construction of ordered groupoids
from categories acting on principal groupoids which we now describe.

Let $C$ be a category and $G$ a groupoid.
Let $\pi \colon \: G \rightarrow C_{o}$ be a function to the set of identities of $C$.
Define
$$C \ast G = \{(a,x) \in C \times G \colon \: \dom (a) = \pi (x)   \}.$$
We say that $C$ {\em acts on} $G$ if there is a function from
$C \ast G$ to $G$, denoted by $(a,x) \mapsto a \cdot x$, which satisfies the axioms (A1)--(A6) below.
Note that I write $\exists a \cdot x$ to mean that $(a,x) \in C \ast G$.
I shall also use $\exists$ to denote the existence of products in the categories
$C$ and $G$.
\begin{description}
\item[{\rm (A1)}] $\exists \pi (x) \cdot x$ and $\pi (x) \cdot x = x$.
\item[{\rm (A2)}] $\exists a \cdot x$ implies that $\pi (a \cdot x) = \ran (a)$.
\item[{\rm (A3)}] $\exists a \cdot (b \cdot x)$ iff $\exists (ab) \cdot x$,
and if they exist they are equal.
\item[{\rm (A4)}] $\exists a \cdot x$ iff $\exists a \cdot \dom (x)$, and if they exist then
$\dom (a \cdot x) = a \cdot \dom (x)$;
$\exists a \cdot x$ iff $\exists a \cdot \ran (x)$, 
and if they exist then $\ran (a \cdot x) = a \cdot \ran (x)$.
\item[{\rm (A5)}] If $\pi (x) = \pi (y)$ and $\exists xy$ then $\pi (xy) = \pi (x)$.
\item[{\rm (A6)}] If $\exists a \cdot (xy)$ then $\exists (a \cdot x)(a \cdot y)$
and $a \cdot (xy) = (a \cdot x)(a \cdot y)$.
\end{description}

We write $(C,G)$ to indicate the fact that $C$ acts on $G$.
If $C$ acts on $G$ and $x \in G$ define
$$C \cdot x = \{a \cdot x \colon \: \exists a \cdot x \}.$$
Define $x \preceq y$ in $G$ iff there exists $a \in C$ such that $x = a \cdot y$.
The relation $\preceq$ is a preorder on $G$.
Let $\equiv$ be the associated equivalence:
$x \equiv y$ iff $x \preceq y$ and $y \preceq x$.
Observe that $x \preceq y$ iff $C \cdot x \subseteq C \cdot y$.
Thus $x \equiv y$ iff $C \cdot x = C \cdot y$.
Denote the $\equiv$-equivalence class containing $x$ by $[x]$,
and denote the set of $\equiv$-equivalence classes by $J(C,G)$.
The set $J(C,G)$ is ordered by $[x] \leq [y]$ iff $x \preceq y$.

\begin{remarks}{\em \mbox{}
\begin{enumerate}

\item Axioms (A1), (A2) and (A3) are the usual axioms for the action of a category on a set.

\item If $x \equiv y$ then ${\bf d}(x) \equiv {\bf d}(y)$ and ${\bf r}(x) \equiv {\bf r}(y)$ by axiom (A4).

\item For each $e \in C_{o}$ put $G_{e} = \pi^{-1}(e)$.
Let $x \in G_{e}$.
Then $\pi (x) = e$.
Thus $\exists e \cdot x$ and so by axiom (A4), we also have that $\exists e \cdot \dom (x)$ and $\exists e \cdot \ran (x)$.
Thus $\dom (x), \ran (x) \in G_{e}$.
Also $\exists \dom (x)$ means that $e \cdot (x^{-1} x)$ and so by axiom (A6),
we have that $\exists e \cdot x^{-1}$ and so $x^{-1} \in G_{e}$.
By axiom (A5), if $x,y \in G_{e}$ and $xy$ is defined then $xy \in G_{e}$.
It follows that $G_{e}$ is a subgroupoid of $G$,
and by axiom (A4) it must be a union of connected components of $G$.

If $x$ is an identity in $G$ and $\exists a \cdot x$ then $a \cdot x$ is an identity in $G$.
This follows by (A4), since ${\bf d}(a \cdot x) = a \cdot {\bf d}(x) = a \cdot x$.
Combining this with axiom (A6), we see that if $f \stackrel{a}{\longleftarrow} e$ in $C$, then the function
$x \mapsto a \cdot x$ from $G_{e}$ to $G_{f}$ is a functor.

\end{enumerate}
}
\end{remarks}

We shall be interested in actions of categories $C$ on groupoids $G$ that satisfy two further conditions:
\begin{description}
\item[{\rm (A7)}] $G$ is principal.
\item[{\rm (A8)}] $\dom (a \cdot x) = \dom (b \cdot x)$ iff $\ran (a \cdot x) = \ran (b \cdot x)$.
\end{description}
Condition (A7) is to be expected;
condition (A8) will make everything work, as will soon become clear.
The axioms (A7) and (A8) together imply that if $\dom (a \cdot x) = \dom (b \cdot x)$ then $a \cdot x = b \cdot x$.

\begin{theorem} Let $C$ be a category acting on the groupoid $G$,
and suppose in addition that both (A7) and (A8) hold.
Then

\begin{enumerate}
\item $J(C,G)$ is an ordered groupoid.
\item $J(C,G)$ is $\ast$-inductive iff for all identities $e,f \in G$ we have that 
$C \cdot e \cap C \cdot f$ non-empty implies there exists an identity $i$ 
such that $C \cdot e \cap C \cdot f = C \cdot i$.
\end{enumerate}

\end{theorem}
\begin{proof}
(1) Define 
$$\dom [x] = [\dom (x)] 
\mbox{ and }
\ran [x] = [\ran (x)].\footnote{Strictly speaking, I should write 
$\dom ([x])$ but I shall omit the outer pair of brackets.}$$
These are well-defined by Remarks~4.1(2).

We claim that $\dom [x] = \ran [y]$ iff
there exists $x' \in [x]$ and $y' \in [y]$ such that $\exists x'y'$.
To prove this, suppose first that $\dom [x] = \ran [y]$.
Then $\dom (x) \equiv \ran (y)$.
There exist elements $a,b \in C$ such that
$\dom (x) = a \cdot \ran (y)$ and $\ran (y) = b \cdot \dom (x)$.
Thus by (A3) and (A4), we have that
$$\ran (b \cdot (a \cdot y)) = \ran (y).$$
By (A8), this implies that 
$$\dom (b \cdot (a \cdot y)) = \dom (y).$$
By (A7), this means that $y = b \cdot (a \cdot y)$.
Hence $y \equiv a \cdot y$ and $\exists x (a \cdot y)$, as required.
The converse follows by Remarks~4.1(2).

We define a partial product on $J(C,G)$ as follows:
if $\dom [x] = \ran [y]$ then 
\begin{center}
$[x][y] = [x'y']$ where $x' \in [x]$, $y' \in [y]$ and $\exists x'y'$,
\end{center}
otherwise the partial product is not defined.
To show that it is well-defined we shall use (A7) and (A8).
Let $x'' \in [x]$ and $y'' \in [y]$ be such that $\exists x''y''$.
We show that $x'y' \equiv x''y''$.
By definition there exist $a,b,c,d \in C$ such that
$$
x' = a \cdot x, 
\quad
x = b \cdot x',
\quad
x'' = c \cdot x,
\quad
x = d \cdot x''
$$
and there exist $s,t,u,v \in C$ such that
$$
y' = s \cdot y,
\quad
y = t \cdot y',
\quad
y'' = u \cdot y,
\quad
y = v \cdot y''.$$
Now
$x =  b \cdot x'$ and $x'' = c \cdot x$.
Thus $x'' = (cb) \cdot x'$ by (A3).
Now $\exists x'y'$ and so 
$\pi (x'y') = \pi (x')$ by (A5).
Thus $\exists (cb) \cdot (x'y')$.
Hence 
$(cb) \cdot (x'y') = [(cb) \cdot x'][(cb) \cdot y']$ by (A6) which is $x''[(cb) \cdot y']$.
We shall show that $(cb) \cdot y' = y''$, which proves that $x''y'' \preceq x'y'$;
the fact that $x'y' \preceq x''y''$ holds by a similar argument so that $x'y' \equiv x''y''$ as required.
It therefore only remains to prove that $(cb) \cdot y' = y''$.
We have that
$y'' = (ut) \cdot y'$ and $\dom (x'') = \ran (y'')$.
Thus
$\dom (x'') = \ran (y'') = (ut) \cdot \ran (y')$ by (A4).
But $\dom (x'') = (cb) \cdot \ran (y')$.
Thus $(ut) \cdot \ran (y') = (cb) \cdot \ran (y')$.
Hence
$$\ran ((ut) \cdot y') = \ran ((cb) \cdot y')$$
by (A4).
By axioms (A7) and (A8), it follows that 
$(cb) \cdot y' = (ut) \cdot y' = y''$
and so the partial product is well-defined.

Thus $J(C,G)$ is a groupoid in which $[x]^{-1} = [x^{-1}]$, and the identities are the elements of the form $[x]$ where $x \in G_{o}$.
The order on $J(C,G)$ is defined by $[x] \leq [y]$ iff $x = a \cdot y$ for some $a \in C$.
It remains to show that $J(C,G)$ is an ordered groupoid with respect to this order.

(OG1) holds: let $[x] \leq [y]$.
Then $x = a \cdot y$.
By axiom (A4), we have that $\dom (x) = a \cdot \dom (y)$.
Thus $x^{-1}x = a \cdot (y^{-1}y)$.
By axiom (A6), $x^{-1}x = (a \cdot y^{-1})(a \cdot y)$.
Similarly $xx^{-1} = (a \cdot y)(a \cdot y^{-1})$.
But $G$ is a principal groupoid and so $x^{-1 } = (a \cdot y^{-1})$ and so $[x^{-1}] \leq [y^{-1}]$, as required.

(OG2) holds: let $[x] \leq [y]$ and $[u] \leq [v]$ and suppose that the partial products $[x][u]$ and $[y][v]$ exist.
Then there exist $x' \in [x]$, $u' \in [u]$, $y' \in [y]$ and $v' \in [v]$ such that
$[x][u] = [x'u']$ and $[y][v] = [y'v']$.
By assumption, $[x'] \leq [y']$ and $[u'] \leq [v']$ so that there exist $a,b \in C$ such that
$x' = a \cdot y'$ and $u' = b \cdot v'$.
We need to show that $x'u' \preceq y'v'$.
Now $\dom (x') = \ran (u')$ and so $a \cdot \dom (y') = b \cdot \ran (v')$.
But $\dom (y') = \ran (v')$.
Thus
$a \cdot \dom (y') = b \cdot \dom (y')$.
Hence
$$\dom (a \cdot y') = \dom (b \cdot y').$$
By (A8), we therefore have that
$$\ran (a \cdot y') = \ran (b \cdot y'),$$
and so $a \cdot y' = b \cdot y'$ by (A7).
Thus
$x'u' = (a \cdot y')(b \cdot v') = (b \cdot y')(b \cdot v')$.
Now $\exists y'v'$ and so by (A5) and (A6) we have that
$(b \cdot y')(b \cdot v') = b \cdot (y'v')$.
Thus $x'u' = b \cdot (y'v')$ and so $x'u' \preceq y'v'$, as required.

(OG3) holds: let $[e] \leq \dom [x]$ where $e \in G_{o}$.
Then $e \preceq \dom (x)$ and so $e = a \cdot \dom (x)$ for some $a \in C$.
Now $\exists a \cdot x$ by (A4).
Define 
$$([x]\,|\,[e]) = [a \cdot x].$$
Clearly $[a \cdot x] \leq [a]$, 
and $\dom [a \cdot x] = [a \cdot \dom (x)] = [e]$.
It is also unique with these properties as we now show.
Let $[y] \leq [x]$ such that $\dom [y] = [e]$.
Then $y = b \cdot x$ for some $b \in C$ and $\dom (y) \equiv e$.
Because of the latter, there exists $c \in C$ such that $e = c \cdot \dom (y)$.
Thus $e = (cb) \cdot \dom (x)$.
But $e = a \cdot \dom (x)$ and so $(cb) \cdot \dom (x) = a \cdot \dom (x)$.
By (A7) and (A8), we therefore have that $c \cdot y = a \cdot x$.
It follows that we have shown that $a \cdot x \preceq y$.
From $\dom (y) \equiv e$, there exists $d \in C$
such that $\dom (y) = d \cdot e$.
Using (A7) and (A8), we can show that $y = d \cdot (a \cdot x)$,
and so $y \preceq a \cdot x$.
We have therefore proved that $y \equiv a \cdot x$.
Hence $[y] = [a \cdot x]$, as required.

(OG3)$^{\ast}$ holds: although this axiom follows from the others,
we shall need an explicit description of the corestriction.
Let $[e] \leq \ran [x]$ where $e \in G_{o}$.
Then $e \preceq \ran (x)$ and so $e = b \cdot \ran (x)$ for some $b \in C$.
Now $\exists b \cdot x$ by (A4).
Define
$$([e]\,|\,[x]) = [b \cdot x].$$
The proof that this has the required properties is similar to the one above.

(2) We now turn to the properties of the pseudoproduct in $J(C,G)$.
Let $[e],[f]$ be a pair of identities in $J(C,G)$.
It is immediate from the definition of the partial order that 
$[e]$ and $[f]$ have a lower bound iff $C \cdot e \cap C \cdot f \neq \emptyset$.
Next, a simple calculation shows that
$[i] \leq [e],[f]$ iff $C \cdot i \subseteq C \cdot e \cap C \cdot f$.
It is now easy to deduce that
$[i] = [e] \wedge [f]$ iff $C \cdot i = C \cdot e \cap C \cdot f$.

It will be useful to have a description of the pseudoproduct itself.
If $C \cdot i = C \cdot e \cap C \cdot f$
then denote by 
\begin{center}
$e \ast f$ and $f \ast e$
\end{center} 
elements of $C$, not necessarily unique, 
such that 
$$i = (e \ast f) \cdot f = (f \ast e) \cdot e.$$
Suppose that $[x],[y]$ are such that the pseudoproduct $[x]\otimes [y]$ exists.
Then by definition $[\dom (x)] \wedge [\ran (y)]$ exists.
Thus $C \cdot \dom (x) \cap C \cdot \ran (y) = C \cdot e$ for some $e \in G_{o}$.
It follows that
$$[x] \otimes [y] = ([x]\,|\,[e])([e]\,|\,[y]).$$ 
Now 
$$([x]\,|\,[e]) = [(\ran (y) \ast \dom (x)) \cdot x]$$
and
$$([e]\,|\,[y]) = [(\dom (x) \ast \ran (y))\cdot y].$$ 
Hence
$$[x] \otimes [y] 
=
[((\ran (y) \ast \dom (x)) \cdot x)
((\dom (x) \ast \ran (y))\cdot y)
].
$$
\end{proof}

The condition that if $C \cdot e \cap C \cdot f$ is non-empty,
where $e$ and $f$ are identities,
then there exists an identity $i$ such that $C \cdot e \cap C \cdot f = C \cdot i$
will be called the {\em orbit condition} for the pair $(C,G)$.
Part (2) of Theorem~4.1 can therefore be stated thus:
$J(C,G)$ is $\ast$-inductive iff $(C,G)$ satisfies the orbit condition.

\subsection{Universality of the construction}

In this section, I shall show that every ordered groupoid is isomorphic to 
one of the form $J(C,H)$ for some action of a category $C$ on a principal groupoid $H$.

Let $G$ be an ordered groupoid.
There are three ingredients needed to construct $J(C,H)$:
a category, which I shall denote by $\mathbf{R}(G)$, a principal groupoid, which I shall denote by $R(G)$, 
and a suitable action of the former on the latter.
We define these as follows:
\begin{itemize}

\item We define the category $\R (G)$ as in Section~3 as follows:
an element of $\R (G)$ is an ordered pair $(x,e)$ 
where $(x,e) \in G \times G_{o}$ and $\dom (x) \leq e$.
We define a partial product on $\R (G)$ as follows:
if $(x,e),(y,f) \in \R (G)$ and $e = \ran (y)$ then
$(x,e)(y,f) = (x \otimes y, f)$.
Thus $\R (G)$ is a right cancellative category with identities $(e,e) \in G_{o} \times G_{o}$.

\item We define the groupoid $R(G)$ as follows:
its elements are pairs $(x,y)$ where $\ran (x) = \ran (y)$.
Define $\dom (x,y) = (y,y)$ and $\ran (x,y) = (x,x)$.
The partial product is defined by $(x,y)(y,z) = (x,z)$.
Evidently, $R(G)$ is the groupoid associated with the equivalence relation
that relates $x$ and $y$ iff $\ran (x) = \ran (y)$.

\item We shall now define what will turn out to be an action of $\R (G)$ on $R(G)$.
Define $\pi \colon \: R(G) \rightarrow \R (G)_{o}$ by $\pi (x,y) = (\ran (x),\ran (y))$,
a well-defined function.
Define $(g,e) \cdot (x,y) = (g \otimes x, g \otimes y)$ iff
$e = \ran (x) = \ran (y)$.
This is a well-defined function from $\R (G) \ast R(G)$ to $R(G)$.
\end{itemize}

\begin{proposition} Let $G$ be an ordered groupoid.
With the above definition, the pair $(\R (G),R(G))$ satisfies axioms (A1)--(A8).
\end{proposition}
\begin{proof}
The verification of axioms (A1)--(A7) is routine.
We show explicitly that (A8) holds.
Suppose that
$$\ran [(s,e) \cdot (x,y)]  = \ran [(t,e) \cdot (x,y)].$$ 
Then $s \otimes x = t \otimes x$.
The groupoid product $x^{-1}y$ is defined,
and the two ways of calculating the pseudoproduct of the triple $(s,x,x^{-1}y)$ 
are defined, and the two ways of calculating the pseudoproduct of the triple
$(t,x,x^{-1}y)$ are defined.
It follows that $s \otimes y = t \otimes y$;
that is,
$$\dom [(s,e) \cdot (x,y)]  = \dom [(t,e) \cdot (x,y)].$$ 
The converse is proved similarly.
\end{proof}

The next theorem establishes what we would hope to be true is true.

\begin{theorem} Let $G$ be an ordered groupoid.
Then $J(\R (G),R(G))$ is isomorphic to $G$.
\end{theorem}
\begin{proof}
Define $\alpha \colon \: G \rightarrow J(\R (G),R(G))$ by $\alpha (g) = [(\ran (g),g)]$.
We show first that $\alpha$ is a bijection.
Suppose that $\alpha (g) = \alpha (h)$.
Then $(\ran (g),g) \equiv (\ran (h),h)$.
Thus 
$(a,\ran (g)) \cdot (\ran (g),g) = (\ran (h),h)$
and
$(b,\ran (h)) \cdot (\ran (h),h)= (\ran (g),g)$
for some category elements $(a,\ran (g))$ and $(b,\ran (h))$.
Hence
$$a \otimes \ran (g) = \ran (h),\quad
b \otimes \ran (h) = \ran (g), \quad
a \otimes g = h,
\mbox{ and }
b \otimes h = g.$$
It follows that $a$ and $b$ are identities
and so $h \leq g$ and $g \leq h$, which gives $g = h$.
Thus $\alpha$ is injective.
To prove that $\alpha$ is surjective,
observe that if $[(x,y)]$ is an arbitrary element of $J(\R (G),R(G))$,
then $(x,y) \equiv (\dom (x), x^{-1}y)$ because
$$(x^{-1},\ran (x))\cdot (x,y) = (\dom (x),x^{-1}y)
\mbox{ and }
(x,\dom (x)) \cdot (\dom (x),x^{-1}y) = (x,y).$$

Next we show that $\alpha$ is a functor.
It is clear that identities map to identities.
Suppose that $gh$ is defined in $G$.
Now $\alpha (g) = [(\ran (g),g)]$ and $\alpha (h) = [(\ran (h),h)]$.
We have that
$\dom [(\ran (g),g)] = [(g,g)]$ and $\ran [(\ran (h),h)] = [(\ran (h), \ran (h))]$.
Now $(g,g) \equiv (\dom (g), \dom (g))$ because
$$(g^{-1},\dom (g)) \cdot (g,g) = (\dom (g),\dom (g))$$
and
$$(g,\dom (g)) \cdot (\dom (g),\dom (g)) = (g,g).$$
Thus $\alpha (g)\alpha (h)$ is also defined.
Now $(\ran (h),h) \equiv (g,gh)$ because
$$(g,\ran (h)) \cdot (\ran (h),h) = (g,gh)$$
and 
$$(g^{-1},\ran (g)) \cdot (g,gh) = (\ran (h),h).$$
Thus $\alpha (g) \alpha (h) = [(\ran (g),gh)] = \alpha (gh)$.
It follows that $\alpha$ is a functor.

Finally, we prove that $\alpha$ is an order isomorphism.
Suppose first that $g \leq h$ in $G$.
Then $g^{-1} \leq h^{-1}$ 
and
$(\dom (g)|h^{-1}) \leq h^{-1}$
and 
$\ran (\dom (g)|h^{-1}) = \dom (g) = \ran (g^{-1})$.
Thus $(\dom (g)|h^{-1}) = g^{-1}$.
It is now easy to check that
$(\ran (g),g) = (g \otimes h^{-1},\ran (h)) \cdot (\ran (h),h)$.
Thus $\alpha (g) \leq \alpha (h)$.
Now suppose that $\alpha (g) \leq \alpha (h)$.
Then $(\ran (g),g) = (a,\ran (h)) \cdot (\ran (h),h)$.
It follows that $a$ is an identity and that $g = a \otimes h$ and so $g \leq h$.
We have proved that $\alpha$ is an order isomorphism.
Hence $\alpha$ is an isomorphism of ordered groupoids.
\end{proof}

Theorem~4.4 tells us that ordered groupoids can be constructed from pairs $(C,G)$ satisfying some additional conditions.
First, we may assume that $C$ is right cancellative.
Second, $\pi \colon G \rightarrow C_{o}$ is a surjective map.
Third if $a \cdot x = b \cdot x$ then $a = b$;
we call this the {\em right cancellation condition}.
A pair satisfying these three additional condition is called an {\em affine system} and are the big sisters to the RP-systems of classical semigroup theory \cite{McA,Reilly}.
In an affine system, the equivalence relation $\equiv$ is determined by the isomorphisms in $C$;
for suppose that $y = a \cdot x$ and $x = b \cdot y$ then
$y = (ab) \cdot y$ and $x = (ba) \cdot x$ and so by the right cancellation condition
$a$ is invertible with inverse $b$.

\subsection{Morphisms between affine systems}

Let $(C,G)$ and $(D,H)$ be affine systems.
A {\em morphism} from $(C,G)$ to $(D,H)$ is a pair $\alpha = (\alpha_{1},\alpha_{2})$
where $\alpha_{1} \colon C \rightarrow D$ and $\alpha_{2} \colon G \rightarrow H$ are functors
such that if $a \cdot x$ is defined then $\alpha_{1} (a) \cdot \alpha_{2} (x)$ is defined and
$\alpha_{2} (a \cdot x) =  \alpha_{1} (a) \cdot \alpha_{2} (x)$. 

\begin{lemma} Let $\alpha \colon (C,G) \rightarrow (D,H)$ be a morphism of affine systems.
Then $\bar{\alpha} \colon J(C,G) \rightarrow J(D,H)$, 
defined by $\bar{\alpha}[x] = [\alpha_{2}(x)]$, 
is an ordered functor.
\end{lemma}
\begin{proof} Observe that if $x \preceq y$ then $\alpha_{2} (x) \preceq \alpha_{2} (y)$.
Thus the function $\bar{\alpha}$ is well-defined.
Since $\alpha_{2}$ is a functor $\bar{\alpha}$ maps identities to identities.
Suppose that $[x][y]$ is defined in $J(C,G)$.
Then there exist $x' \equiv x$ and $y' \equiv y$ such that $x'y'$ is defined in $G$.
Thus $\alpha_{2}(x)\alpha_{2}(y)$ is defined in $H$.
We have that 
$$\bar{\alpha}([x][y]) = \bar{\alpha}[x'y'] = [\alpha_{2}(x'y')] = [\alpha_{2}(x')\alpha_{2}(y')] = [\alpha_{2}(x')][\alpha_{2}(y')]$$
which is just 
$\bar{\alpha}[x] \bar{\alpha}[y]$.
Thus $\bar{\alpha}$ is a functor and it is an ordered by our first observation.
\end{proof}

A {\em morphism} $\alpha = (\alpha_{1},\alpha_{2})$ from $(C,G)$ to $(D,H)$ is said to be an {\em equivalence} if 
the following three conditions hold:
\begin{description}
\item[{\rm (E1)}] $\alpha_{1}$ is an equivalence of categories.
\item[{\rm (E2)}] $\alpha_{2} (x) \preceq \alpha_{2}(y)$ implies $x \preceq y$.
\item[{\rm (E3)}] For each $y \in H$ there exists $x \in G$ such that $y \equiv \alpha_{2} (x)$.
\end{description}

\begin{lemma} Let $\alpha \colon (C,G) \rightarrow (D,H)$ be an equivalence of affine systems.
Then $\bar{\alpha} \colon J(C,G) \rightarrow J(D,H)$ is an isomorphism of ordered groupoids.
\end{lemma}
\begin{proof} 
(E2) implies that $\bar{\alpha}$ is injective and (E3) that $\bar{\alpha}$ is surjective.
The condition (E2) also implies that $\bar{\alpha}$ is an order isomorphism.
\end{proof}

\subsection{Special cases and examples}

We sketch out a few special cases of the above theory and touch on some interesting examples.

We show first that the theory of ordered groupoids with a maximum identity, described in Section~3.4, is a special case of this new theory.
The proof of the following is immediate from the definitions.
It is not surprising given the motivation described at the beginning of Section~4.1.

\begin{proposition} Let $C$ be a right cancellative category with weak initial identity 1.
Put $X = C1$ and $X \ast X = \{(a,b) \in X \times X \colon \ran (a) = \ran (b) \}$.
Define an action of $C$ on $X \ast X$ by $a \cdot (x,y) = (ax, ay)$ if $\dom (a) = \ran (x)$.
Then $(C,X \ast X)$ is an affine system and $\mathbf{G}^{r}(C)$, the ordered groupoid constructed from
the right rooted category $C$ is isomorphic to $J(C,X \ast X)$.
\end{proposition}

An affine system $(C,G)$ is said to be {\em cyclic} if the following two conditions hold:
\begin{description}

\item[{\rm (C1)}] There exists $x_{0} \in G_{o}$ such that $G_{o} = C \cdot x_{0}$.

\item[{\rm (C2)}] If $a,b \in C$ such that $\ran (a) = \ran (b)$ and $\dom (a) = \dom (b) = \pi (x_{0})$ then
there exists $g \in G$ such that $\ran (g) = a \cdot x_{0}$ and $\dom (g) = b \cdot x_{0}$. 

\end{description}

Ordered groupoids with maximum identity correspond to cyclic affine systems.

Our second special case deals with the situation where our ordered groupoid is connected.
We consider an affine system $(C,G)$ where $C$ is a right cancellative monoid and $G = X \times X$ is a universal principal groupoid.
It follows that in fact we have a left monoid action of $C$ on $X$ which satisfies the right cancellation condition.
We call $(C,X)$ an {\em affine monoid system}.

\begin{proposition} 
Affine monoid systems describe connected ordered groupoids.
\end{proposition}

Affine systems lead to a natural description of arbitrary inverse semigroups with zero.
Let $S$ be an inverse semigroup with zero and let $\mathbf{R} = \mathbf{R}(S^{\ast})$,
the right cancellative category associated with the non-zero elements of $S$.
The principal groupoid $R = R(S^{\ast})$ consists of those pairs $(s,t)$ such that
$s$ and $t$ are both non-zero and $ss^{-1} = tt^{-1}$.
By Theorem~4.4, the ordered groupoid $S^{\ast}$ is isomorphic to $J(\mathbf{R},R)$.
Thus the inverse semigroup $S$ is isomorphic to $J(\mathbf{R},R)^{0}$ equipped with the pseudoproduct.
We may summarize these results as follows.

\begin{theorem} Every inverse semigroup with zero $S$ is determined upto isomorphism by three ingredients:
the right cancellative category $\mathbf{R}(S^{\ast})$, Green's $\mathcal{R}$-relation on the non-zero elements, 
and the action of the category on the groupoid determined by Green's $\mathcal{R}$-relation. 
\end{theorem}

One obvious question is how categories acting on principal groupoids arise.
We now describe one example.
Let $(C,X)$ be a pair consisting of a category $C$ acting on a set $X$
where we denote by $\pi \colon \: X \rightarrow C_{o}$ the function used in
defining the action.
Define the relation $\mathcal{R}^{\ast}$ on the set $X$ as follows:
$x \,\mathcal{R}^{\ast}\, y$ iff $\pi (x) = \pi (y)$ and for all $a,b \in C$
we have that, when defined,
$$a \cdot x = b \cdot x \Leftrightarrow a \cdot y = b \cdot y.$$
Observe that $\mathcal{R}^{\ast}$ is an equivalence relation on the set $X$.
In addition, $x \,\mathcal{R}^{\ast}\, y$ implies that
$c \cdot x \,\mathcal{R}^{\ast}\, c \cdot y$
for all $c \in C$ where $c \cdot x$ and $c \cdot y$ are defined.
Consequently, we get a principal groupoid 
$$G(C,X) = \{(x,y) \colon \: x \,\mathcal{R}^{\ast}\, y \}.$$
Define $\pi' \colon \: G(C,X) \rightarrow C_{o}$ by
$\pi' (x,y) = \pi (x)$, 
and define an action of $C$ on $G(C,X)$ by 
$a \cdot (x,y) = (a \cdot x, a \cdot y)$
when ${\bf d}(a) = \pi' (x,y)$.
It is easy to check that axioms (A1)--(A8) hold.
We may therefore construct an ordered groupoid from the pair
$(C,G(C,X))$.
This is identical to the ordered groupoid constructed in
\cite{Lawson99} directly from the pair $(C,X)$.

Constructing inverse semigroups from rooted categories and affine systems has practical applications.
One can try to relate the categorical properties of the rooted category to the algebraic properties of the inverse monoid \cite{Leech98}.
It also leads to the perspective that inverse semigroup theory can be viewed not just as the abstract theory
of partial bijections but also the abstract theory of bijections between quotients \cite{F,FL}.
The theory restricted to a class of $0$-bisimple inverse monoids allows self-similar group actions \cite{Nek} to be
described in terms of a class of inverse semigroups \cite{Lawson08}.
Graph inverse semigroups can be constructed from free categories \cite{Lawson99}.
Such inverse semigroups are important in the theory of $C^{\ast}$-algebras \cite{Raeburn}.
Finally, general affine systems \cite{Lawson06} can be used to better understand the nature of the inverse semigroups that Dehornoy first constructed \cite{Dehornoy93}.

\section{Notes on Chapter 2}

The theory described in Section~2 originates in Ehresmann's work \cite{Ehresmann} and was first developed
within inverse semigroup by Boris Schein \cite{Schein} and then successfully generalized 
to arbitrary regular semigroups by Nambooripad \cite{Nam}.
The maximum enlargement theorem comes under scrutiny in \cite{Steinberg} and \cite{Lawson02a}.
I wrote extensively about the ordered groupoid approach to inverse semigroup theory in my book \cite{Lawson98} so I will say no more here.

The origins of Section~3 go right back to the beginnings of inverse semigroup theory.
Clifford \cite{Clifford} showed how to describe inverse monoids in terms of left cancellative monoids.
The realization that this could be generalized to arbitrary inverse monoids by replacing left cancellative monoids by 
left cancellative categories is due to Leech \cite{Leech} who gives duw credit to Logananthan's trail-blazing paper \cite{Log}
on how category theory can be applied to semigroup theory;
Loganathan showed that the cohomology of an inverse semigroup,
introduced by Lausch \cite{Lausch}, was the same as the cohomology of the associated left rooted category.

The extension of Clifford's original construction to $0$-bisimple inverse monoids is due to Reilly \cite{Reilly} and McAlister \cite{McA}.
The routine extension of Leech's construction to inverse monoids with zero may be found in \cite{Lawson99a}.
Note that the fact that this paper is labelled `I' and that the references refer to two further papers by the same author labelled `II' and `III' can safely be ignored.
The more general results on ordered groupoids were proved in \cite{Lawson04} and \cite{Lawson04b}.
In Section~2.4, I described how ordered groupoids could be used to find a structural description of $E$-unitary inverse semigroups.
In \cite{James}, the authors show how the left rooted category associated with an $E$-unitary
inverse monoid has a groupoid of fractions from which a proof of the $P$-theorem can also be deduced.
There are almost as many proofs of the $P$-theorem as there are semigroup theorists and a survey of some of them can be found in \cite{Lawson07}.
Chapter~X of Petrich \cite{Petrich} contains much more on the theory of bisimple inverse monoids.

Section~4 arises as unfinished business from Section~3.
We know how to construct inverse monoids with zero from left rooted categories but the question remains of how to deal with the semigroup case.
McAlister \cite{McA} and Reilly \cite{Reilly} discovered how to deal with bsimple inverse smeigroups and $0$-bisimple inverse semigroups
by using what were called {\em RP-systems} and {\em generalized RP-systems}.
As a result of reading a paper by Girard on linear logic \cite{Girard},
I was led to the construction described in \cite{Lawson99},
which shows how inverse semigroups can be constructed from 
categories acting on {\em sets}.
Claas R\"over pointed out to me that the paper by Dehornoy \cite{Dehornoy93}.
Dehornoy constructs an inverse semigroup from any variety,
in the sense of universal algebra,
that is described by equations which are balanced,
meaning that the same variables occur on either side of the equation.
This construction was clearly related to the construction in \cite{Lawson99}
and an analysis of the connections between the two that led to \cite{Lawson05}
and the {\em affine systems} described in this section.
This raises the obvious question: what have inverse semigroups got to do with linear logic?



\begin{thebibliography}{99}

\bibitem{Clifford} A.~Clifford, A class of d-simple semigroups,
{\em Amer.~J.~Math.} {\bf 75} (1953), 547-556.

\bibitem{Dehornoy93} P.~Dehornoy, Structural monoids associated to equational varieties,
{\em Proc. Amer. Math. Soc.} {\bf 117} (1993), 293--304. 

\bibitem{Ehresmann} C.~Ehresmann, {\em Oeuvres compl\`etes et comment\'ees},
(ed A.~C.~Ehresmann) Supplements to {\em Cah. de Topol. G\'eom. Diff\'er. Cat\'eg.} Amiens, 1980-83.

\bibitem{F} D.~G.~FitzGerald, A presentation for the monoid of uniform block permutations,
{\em Bull. Austral. Math. Soc.} {\bf 68} (2003), 317--324.

\bibitem{FL} D.~G.~FitzGerald, J.~Leech, Dual symmetric inverse monoids and representation theory,
{\em J.~Austral.~Math.~Soc. Ser~A.} {\bf 64} (1998), 345--367.

\bibitem{Girard} J.-Y.~Girard, The geometry of interaction III: accommodating the additives,
in {\em Advances in linear logic} (eds J.-Y. Girard, Y.~Lafont, L.~Regnier) Cambridge University Press, 1995.

\bibitem{James} H.~James, M.~V.~Lawson, An application of groupoids of fractions to inverse semigroups,
{\em Periodica Math. Hung.} {\bf 38} (1999), 43--54.

\bibitem{Lausch} H.~Lausch, Cohomology of inverse semigroups,
{\em J.~Algebra} {\bf 35} (1975), 273--303.

\bibitem{Lawson98} M.~V.~Lawson, {\em Inverse semigroups: the theory of partial symmetries}, World Scientific, 1998.

\bibitem{Lawson99} M.~V.~Lawson, Constructing inverse semigroups from category actions,
{\em J. Pure Appl. Alg.} {\bf 137} (1999), 57--101.

\bibitem{Lawson99a} M.~V.~Lawson, The structure of $0$-$E$-unitary inverse semigroups I:
the monoid case, {\em Proc. Edinb. Math. Soc.} {\bf 42} (1999), 497--520.

\bibitem{Lawson02a} M.~V.~Lawson, J.~Matthews, T.~Porter, The homotopy theory of inverse semigroups,
{\em Internat. J. Algebra Comput.} {\bf 12} (2002), 755--790.

\bibitem{Lawson02b} M.~V.~Lawson, $E^{\ast}$-unitary inverse semigroups,
in {\em Semigroups, algorithms, automata and languages} (eds G.~M.~S.~Gomes, J.-E.~Pin, P.~V.~Silva)
World Scientific, 2002, 195--214.

\bibitem{Lawson04} M.~V.~Lawson, Ordered groupoids and left cancellative categories, {\em Semigroup Forum}
{\bf 68} (2004), 458--476.

\bibitem{Lawson05} M.~V.~Lawson, Constructing ordered groupoids, {\em Cah. de Topol. G\'eom. Diff\'er. Cat\'eg.} {\bf 46} (2005), 123--138.

\bibitem{Lawson06} M.~V.~Lawson, A correspondence between balanced varieties and inverse monoids, 
{\em Inter. J. Algebra Comput.} {\bf 16} (2006), 887--924.

\bibitem{Lawson08} M.~V.~Lawson, A correspondence between a class of monoids and self-similar group actions I, 
{\em Semigroup Forum} {\bf 76} (2008), 489-517.

\bibitem{Lawson07} M.~V.~Lawson, S.~W.~Margolis, In McAlister's footsteps: a random ramble around the $P$-theorem,
in {\em Semigroups and formal languages} (eds Andr\'e, Fernandes, Branco, Gomes, Fountain, Meakin),
World Scientific, (2007), 145--163.

\bibitem{Lawson04b} M.~V.~Lawson, B.~Steinberg, Etendues and ordered groupoids, {\em Cah. de Topol. G\'eom. Diff\'er. Cat\'eg.}
{\bf 45} (2004), 82--108.

\bibitem{Leech} J.~Leech, Constructing inverse semigroups from small categories,
{\em Semigroup Forum} {\bf 36} (1987), 89--116.

\bibitem{Leech98} J.~Leech, On the foundations of inverse monoids and inverse algebras, 
{\em Proc. Edinb. Math. Soc.} {\bf 41} (1998), 1--21.


\bibitem{Log} M.~Loganathan, Cohomology of inverse semigroups,
{\em J.~Algebra} {\bf 70} (1981), 375--393.

\bibitem{McA} D.~B.~McAlister, $0$-bisimple inverse semigroups, {\em Proc. Lond. Math. Soc.} {\bf 28} (1974), 193--221.

\bibitem{McA76} D.~B.~McAlister, One-to-one partial right translations of a right cancellative semigroup,
{\em J. Algebra} {\bf 43} (1976), 231--251.

\bibitem{Nam} K.~S.~S.~Nambooripad, {\em Structure of regular semigroups I}, Memoirs of the American
Mathematical Society {\bf 224} (1979).

\bibitem{Nek} V.~Nekrashevych, {\em Self-similar groups}, American Mathematical Society, 2005.

\bibitem{Petrich} M.~Petrich, {\em Inverse semigroups}, John Wiley \& Sons, 1984.

\bibitem{Raeburn} I.~Raeburn, {\em Graph algebras}, AMS, Regional Conference Series in Mathematics, Number~103, 2005. 

\bibitem{Reilly} N.~R.~Reilly, Bisimple inverse semigroups, {\em Trans. Amer. Math. Soc.}
{\bf 132} (1968), 101--114.

\bibitem{Schein} B.~M.~Schein, On the theory of inverse semigroups and generalized grouds,
{\em Amer. Math. Soc. Transl. (2)} {\bf 113} (1979), 89--122.

\bibitem{Steinberg} B.~Steinberg, Factorization theorems for morphisms of ordered groupoids and
inverse semigroups, {\em Proc. Edinb. Math. Soc.} {\bf 44} (2001), 549-569.

\end{thebibliography}
\end{document}